\documentclass{amsart}

\usepackage{fullpage}
\usepackage{colonequals}
\usepackage[alphabetic]{amsrefs}
\usepackage{enumitem}
\usepackage{xcolor,tikz}
\usetikzlibrary[positioning]
\usepackage{hyperref}

\usepackage[normalem]{ulem}
\hypersetup{
    colorlinks=true,
    linkcolor=blue,
    citecolor=blue,
    filecolor=magenta,      
    urlcolor=cyan,
    pdftitle={Overleaf Example},
    pdfpagemode=FullScreen,
    }

\newcommand{\set}[1]{\{#1\}}
\newcommand{\N}{\mathbb{N}}

\newtheorem{theorem}{Theorem}[section]
\newtheorem{lemma}[theorem]{Lemma}

\newtheorem{corollary}[theorem]{Corollary}
\newtheorem{claim}[theorem]{Claim}

\newtheorem*{theorem*}{Theorem}
\newtheorem*{conjecture*}{Conjecture}

\allowdisplaybreaks
\theoremstyle{definition}

\newtheorem{conjecture}[theorem]{Conjecture}

\theoremstyle{remark}
\newtheorem{remark}[theorem]{Remark}

\numberwithin{equation}{section}

\newcommand{\abs}[1]{\lvert#1\rvert}

\makeatletter
\@namedef{subjclassname@2020}{%
  \textup{2020} Mathematics Subject Classification}
  \makeatother

\title{Counting subgraphs of coloring graphs}

\author{Shamil Asgarli}
\address{Department of Mathematics and Computer Science \\ Santa Clara University \\ 500 El Camino Real \\ Santa Clara, CA 95053}
\email{sasgarli@scu.edu}

\author{Sara Krehbiel}
\address{Department of Mathematics and Computer Science \\ Santa Clara University \\ 500 El Camino Real \\ Santa Clara, CA 95053}
\email{skrehbiel@scu.edu}

\author{Howard W. Levinson}
\address{Department of Computer Science \\ Oberlin College \\ 10 N Professor St \\ Oberlin, OH 44074}
\email{hlevinso@oberlin.edu}

\author{Heather M. Russell}
\address{Department of Mathematics and Statistics \\ University of Richmond \\ 410 Westhampton Way \\ Richmond, VA 23173}
\email{hrussell@richmond.edu}

\begin{document}

\begin{abstract}
The chromatic polynomial $\pi_{G}(k)$ of a graph $G$ can be viewed as counting the number of vertices in a family of coloring graphs $\mathcal C_k(G)$ associated with (proper) $k$-colorings of $G$ as a function of the number of colors $k$. These coloring graphs can be understood as a reconfiguration system. We generalize the chromatic polynomial to $\pi_G^{(H)}(k)$, counting occurrences of arbitrary induced subgraphs $H$ in these coloring graphs, and we prove that these functions are polynomial in $k$. In particular, we study the {\em chromatic pairs polynomial} $\pi_{G}^{(P_2)}(k)$, which counts the number of edges in coloring graphs, corresponding to the number of pairs of colorings that differ on a single vertex. We show two trees share a chromatic pairs polynomial if and only if they have the same degree sequence, and we conjecture that the chromatic pairs polynomial refines the chromatic polynomial in general.  
We also instantiate our polynomials with other choices of $H$ to generate new graph invariants.
\end{abstract}

\maketitle

\vspace{1em}


\section{Introduction}
For positive integer $k$, a (proper) {\em $k$-coloring} of a graph $G$ is an assignment of each vertex in $G$ to a number in $\set{1,\dots,k}$ such that no two adjacent vertices have the same color. The {\em chromatic polynomial} $\pi_{G}(k)$ counts the number of $k$-colorings of $G$ as a function of $k$. Birkhoff originally introduced this function to tackle the $4$-color conjecture and showed that $\pi_{G}(k)$ is a polynomial function in $k$ \cite{Bir12}. Read's article  on the chromatic polynomial \cite{Rea68} is an excellent introduction to this important function and graph invariant.

The function $\pi_G(k)$ satisfies the \emph{deletion-contraction} recurrence: $\pi_{G}(k)=\pi_{G-e}(k)-\pi_{G/e}(k)$ for any graph $G$ and edge $e\in E(G)$. By induction on the number of edges, the recurrence implies that $\pi_G(k)$ is a polynomial. This graph invariant contains information such as the number of vertices, edges, triangles, and connected components of $G$.  However, it is far from a complete invariant. For example, every tree $T$ on $n$ vertices shares the chromatic polynomial $\pi_T(k)=k(k-1)^{n-1}$.

\subsection{Our Results}
Noting that the chromatic polynomial counts colorings in isolation, we examine the relationship between related colorings. In particular, we call two colorings adjacent if one can be obtained from the other by toggling the color of a single vertex. Counting the number of such pairs of colorings, we obtain a new invariant, which we call the {\em chromatic pairs polynomial}. We show that this function takes a particular form for trees (see Table~\ref{tab:cpp}) that immediately implies that two trees with the same degree sequence agree on this function. Conversely, the degree sequence can be extracted from this function 
to show that two trees share a chromatic pairs polynomial if and only if they have the same degree sequence (Theorem~\ref{thm:treedeg}).

Our main result is a generalization of this function that allows us to count more complex relationships between colorings. More precisely, we consider an object called a \emph{$k$-coloring graph}, which organizes all proper $k$-colorings of $G$ in a single graph $\mathcal{C}_k(G)$. Each vertex of $\mathcal{C}_k(G)$ corresponds to a $k$-coloring of $G$, and the edge set of $\mathcal{C}_k(G)$ corresponds to the pairs of adjacent colorings. From this point of view, the chromatic polynomial and chromatic pairs polynomial count the number of vertices and edges, respectively, of the coloring graph. Our generalization provides a formula for $\pi_{G}^{(H)}(k)$, counting the number of times an arbitrary graph $H$ appears as an induced subgraph of $G$ in $\mathcal{C}_k(G)$ as a function of $k$. Our main result, Theorem~\ref{thm:general} in Section~\ref{sect:general}, is summarized as follows:
\begin{theorem*}[Theorem~\ref{thm:general}]
    Fix graphs $G$ and $H$, and let $\pi_{G}^{(H)}(k)$ denote the number of induced subgraphs of $\mathcal C_k(G)$ that are isomorphic to $H$. Then $\pi_{G}^{(H)}(k)$ is a polynomial in $k$ for $k$ sufficiently large relative to $H$.
\end{theorem*}

We note that $\pi_G^{(P_2)}(k)$ yields the chromatic pairs polynomial, and further exploration of this special case reveals additional insight about how it may be used as a refined invariant relative to the chromatic polynomial. In Section~\ref{sec:instantiations}, we instantiate the chromatic pairs polynomial for the classes of graphs in Table~\ref{tab:cpp}.

\begin{table}[ht]
    \begin{center}
{\renewcommand{\arraystretch}{1.5}{
\begin{tabular}{ |c|c|c| } 
 \hline
Graph &  Chromatic Polynomial & Chromatic Pairs Polynomial \\ 
 \hline
$N_n$ &  $k^n$ & $n\binom{k}{2} k^{n-1}$ \\ 
$K_n$ & $k(k-1)(k-2)\cdots (k-n+1)$ & $\frac{n}{2} k(k-1)(k-2)\cdots (k-n)$ \\
Tree & $k(k-1)^{n-1}$ & $\binom{k}{2}\sum_{v\in V(T)}(k-2)^{\deg(v)}(k-1)^{n-\deg(v)-1}$ \\ 
$C_n$ & $(k-1)^{n} + (-1)^{n} (k-1)$ & $\frac n 2 k(k-4)(k-1)^{n-1}+2n(k-1)^{n-1}+(-1)^n n(k-1)(k-2)$  \\
 \hline
\end{tabular}}}
\end{center}
\caption{Comparison of chromatic polynomials and chromatic pairs polynomials for various graphs}\label{tab:cpp}
\end{table}

In Section~\ref{sect:invariants}, we derive explicit formulas for several coefficients in the chromatic pairs polynomial. 

\begin{theorem*}[Theorem~\ref{thm:lowcoeffs}]
 Let $G$ be a graph with $n$ vertices, $m$ edges, degree sequence $\set{d_i}_{i\in[n]}$, $\ell$ triangles, and $t$ connected components. Then there exist positive rational numbers $a_{n-2},a_{n-3},\dots,a_t$ such that 
    \begin{align*}
 \pi_{G}^{(P_2)}(k) &=  \frac{n}{2} k^{n+1} - \frac{n+nm+2m}{2} k^{n} + \frac{1}{2} \left( \frac{nm(m+1)}{2}+2m^2-m-(n+3)\ell+\frac{1}{2} \sum_{i=1}^{n} d_i^2 \right) k^{n-1}  \\ 
 & \quad - a_{n-2}k^{n-2}+ a_{n-3}k^{n-3} - \ldots + (-1)^{n+1-t} a_{t} k^t\ .
    \end{align*}
\end{theorem*}

Collectively, this gives us results analogous to chromatic uniqueness. Specifically, we show that chromatic pairs polynomials for null and complete graphs, certain types of trees, and cycles are unique. All of these results point to the conjecture that the chromatic pairs polynomial is a refined invariant of the chromatic polynomial in the following sense:
\begin{conjecture*}[Conjecture~\ref{conj:edge-implies-chromatic}]\label{conj1:intro}
    If two graphs have the same chromatic pairs polynomial, then they have the same chromatic polynomial.
\end{conjecture*}

We also study additional instantiations of our general result, including the task of counting cycles in coloring graphs. This task is increasingly difficult for cycles of growing length. We also study necessary and sufficient conditions for occurrences of induced hypercubes in coloring graphs, which relates to independent sets and can be used to distinguish graphs based on coarse information about their coloring graphs. 

Failing to identify non-isomorphic graphs that match on all $H$-polynomials, we predict that a graph $G$ is completely determined by the infinite family of polynomials $\pi_G^{(H)}(k)$ as $H$ ranges through all graphs. 

\begin{conjecture*}[Conjecture~\ref{conj:col-graphs-determine-G}]\label{conj2:intro}
For graphs $G_1$ and $G_2$, $G_1\cong G_2$ if and only if $\pi_{G_1}^{(H)}(k)=\pi_{G_1}^{(H)}(k)$ for every~$H$.
\end{conjecture*}

\subsection{Related works} 

This paper contributes to the broad literature on graph invariants and specifically those works studying generalizations and refinements of the chromatic polynomial. Core to our construction of $\pi_G^{(H)}(k)$ is the notion of {\em restrained chromatic polynomials} due to Erey \cite{Ere15}. This generalization of the chromatic polynomial counts the number of proper colorings of a graph $G$ where each vertex is restrained by a finite list of forbidden colors. Our polynomial $\pi_G^{(H)}(k)$ is a linear
combination of these restrained chromatic polynomials. Central to Erey's work on restrained chromatic polynomials is the proof that they are, in fact, polynomials for large enough $k$. We rely heavily on this result in proving the eventual polynomiality of our function $\pi_G^{(H)}(k)$. 

Our work is inspired in part by coloring graphs $\mathcal{C}_k(G)$. Coloring graphs generalize the chromatic polynomial in the sense that $\pi_G(k) = |V(\mathcal{C}_k(G))|$. These objects were initially studied in the context of counting and sampling colorings \cites{DFFV06, Mol04, Vig00}. For these applications, it is important for $\mathcal{C}_k(G)$ to be connected, and work has been done to establish bounds on $k$ that guarantee connectedness \cite{CVJ08}. More recently, other features of coloring graphs have been explored, including cycles, cut vertices, and block structures \cites{ABFR18, BFHRS16, BKMRSSX23, BBFKKR19}. By counting instances of induced copies of $H$ within a $k$-coloring graph, $\pi_G^{(H)}(k)$ seeks to translate structural information about the coloring graph into the language of polynomial functions. 

Coloring graphs are one example of the more general notion of reconfiguration graphs. In reconfiguration problems, one considers a set of states of some system equipped with a specific transition rule. From this setup, we construct a reconfiguration graph in which the vertices are states of the system and edges connect states that differ by a single application of a transition rule. Examples  abound from reconfiguration systems for dominating sets ~\cites{ABCHMSS21, HS14}   and shortest paths ~\cite{AEHHNW18}  to reconfiguration systems encoding possible positions of a robotic arm within a restricted space \cites{AG04}. Analogs of $\pi_G^{(H)}(k)$ in other reconfiguration contexts might provide interesting statistics for these systems.

There are many generalizations of the chromatic polynomial, including the important Tutte polynomial (\cite{tutte1954}, see \cite{Bol98}*{Chapter~10}). This two-variable graph invariant appears in other fields, such as knot theory, where one of its specializations yields the famed Jones polynomial for a certain family of knots. The Tutte polynomial can be defined in various ways including via a deletion-contraction process. While the Tutte polynomial is indeed a rich invariant, it exhibits some of the same deficiencies as the chromatic polynomial. In particular, it cannot distinguish between any pair of trees with the same number of edges. In contrast, our chromatic pairs polynomial $\pi_G^{(P_2)}(k)$ can distinguish between trees with distinct degree sequences. 

In another direction, the chromatic symmetric function  \cite{Sta95} and the non-commutative chromatic symmetric function \cite{GS01} extend the chromatic polynomial by encoding each coloring as a monomial with a different variable representing each color. In the non-commutative case, this function provides a complete invariant as it is essentially a list of all possible colorings of the graph (see \cite{GS01}*{Proposition~8.1}). As mentioned above, we conjecture (in Conjecture~\ref{conj:col-graphs-determine-G}) that the family of polynomials $\pi_G^{(H)}(k)$ also provides a complete invariant without retaining such granular coloring data.

\subsection{Structure of the paper} In Section~\ref{sect:prelims}, we streamline notation and discuss preliminary tools, including restrained chromatic polynomials. In Section~\ref{sect:general}, we introduce the concept of an $H$-generator and prove Theorem~\ref{thm:general}, stating that for any graphs $H$ and $G$, the function $\pi_G^{(H)}(k)$, counting the number of occurrences of $H$ as an induced subgraph of $\mathcal C_k(G)$, grows polynomially with $k$. Section~\ref{sec:instantiations} explores $\pi_{G}^{(H)}$ for several choices of $H$, focusing on the cases when $H$ is a 2-path or a cycle. In particular, we compute the chromatic pairs polynomials ($H=P_2$) for null graphs, complete graphs, trees, cycles, and pseudotrees. Section~\ref{sect:invariants} showcases how these polynomials $\pi_{G}^{(H)}$ serve as invariants for various graph families. For example, Theorem~\ref{thm:treedeg} proves that the chromatic pairs polynomial of a tree determines its degree sequence, and Theorem~\ref{thm:cycleuniqueness} establishes the uniqueness of the chromatic pairs polynomial for the cycle graph. Section~\ref{sect:closing-conj} discusses our final conjecture, which predicts that the infinite set of coloring graphs is a complete graph invariant, and we show how this conjecture is equivalent to another conjecture written in terms of our $H$-polynomials.

\medskip 

\textbf{Update on our conjectures.} In response to a posted preprint of this work, Hogan, Scott, Tamitegama, and Tan~\cite{HSTT24} prove new results that resolve our Conjecture~\ref{conj:finite-col-graphs-determine-G}. This recent work also gave an alternative proof of our main result (Theorem~\ref{thm:general}). Their proof did not rely on restrained chromatic polynomials and showed that $\pi_{G}^{(H)}(k)$ is a polynomial for all $k\geq 1$.
Our original preprint established that $\pi_G^{(H)}(k)$ is polynomial for sufficiently large $k$.  We have left our original theorem statement and proof, and we have now added Remark~\ref{remark:thmstatement} after Theorem~\ref{thm:general} observing that our function is in fact also polynomial for all $k\ge 1$. 

\section{Preliminaries}\label{sect:prelims}

We use $\N=\set{1,2,\dots}$ to denote the set of positive integers, and for $k\in \N$, we let $[k]=\set{1,\dots,k}$. For integers $n\ge k\ge 0$, let $\binom{n}{k}=\frac{n!}{k!(n-k)!}$. {For simplicity, let $\binom{n}{k}=0$ when $k>n\ge 0$.}

A (simple, finite) \emph{graph} $G$ consists of a finite vertex set $V(G)$ and edge set $E(G)$, where $E(G)$ is a subset of (unordered) pairs of elements in $V(G)$. For $u,v\in V(G)$, we often use $uv$ as shorthand for $\set{u,v}$  with the understanding that $uv=vu$ for undirected edges. For $v\in V(G)$, we let $N(v)=\set{u\in V(G):uv\in E(G)}$ denote the \emph{neighborhood} of $v$, and we let $\deg(v)=\abs{N(v)}$ denote the \emph{degree} of $v$. For any $U\subseteq V(G)$, we let $G[U]$ denote the subgraph of $G$ \emph{induced} by $U$, that is, the graph with vertex set $U$ and an edge for every edge in $G$ with both ends in $U$. We also write $G-U$ as shorthand for $G[V-U]$ and write $G-v$ as shorthand for $G[V-\set{v}]$ for some $v\in V(G)$. 

In this paper, we will consider several well-studied families of graphs. These families include {null graphs}, where $N_n$ denotes the edgeless $n$-vertex graph; {complete graphs}, where $K_n$ denotes the $n$-vertex graph with $\binom{n}{2}$ edges; {trees}, which are connected graphs with $\abs{E(G)}=\abs{V(G)}-1$; {pseudotrees} (also known as unicyclic graphs), which are trees augmented by one edge; {paths}, where $P_n$ denotes the $n$-vertex tree with max degree 2; {cycle graphs}, where $C_n$ denotes the $n$-vertex graph that is connected and 2-regular; and {hypercube graphs}, most easily defined in terms of Cartesian products. For two graphs $G_1$ and $G_2$, we let $G_1\ \square\ G_2$ denote the Cartesian product of $G_1$ and $G_2$, with vertex set $V(G_1)\times V(G_2)$ and edge set $\set{\set{u_1 u_2,v_1 v_2}\mid u_1=v_1\text{ and }u_2v_2\in E(G_2)\text{, or }u_2=v_2\text{ and }u_1 v_1\in E(G_1)}$. The $d$-dimensional hypercube~$Q_d$ is the Cartesian product of $d$ copies of $P_2$, which has $2^d$ vertices and $d\cdot 2^{d-1}$ edges. 

For $e\in E(G)$, we let $G-e$ denote the graph with vertex set $V(G)$ and edge set $E(G)-\set{e}$, and $G/e$ denotes the graph contracted on edge $e$, in which $e$ is removed and its ends are replaced with a single vertex adjacent to every vertex in the neighborhood of either end of $e$ in $G$. For graph $G$ and $k\in \mathbb N$, a (proper) $k$-coloring of $G$ is a function $c:V(G)\to [k]$ such that $c(u)\ne c(v)$ for all $uv\in E(G)$. The \emph{chromatic polynomial} of $G$ is defined as 
\[
\pi_G(k) = \text{\# distinct $k$-colorings of $G$}
\]
The fact that $\pi_G(k)$ is a polynomial in $k$ can be derived from the \emph{deletion-contraction} principle, which observes that $\pi_{G}(k)=\pi_{G-e}(k)-\pi_{G/e}(k)$ for any graph $G$ and edge $e\in E(G)$ \cite{Rea68}*{Theorem~1}. 

For graph $G$, a \emph{restraint} $r$ is a mapping from $V(G)$ to finite subsets of $\mathbb{N}$. That is, the function $r$ maps each vertex of $G$ to a particular finite subset of $\mathbb{N}$. Fixing graph $G$ and restraint $r$, we define a corresponding \emph{restrained chromatic polynomial} as 
\[
\rho_{G,r}(k)=\abs{\set{c:V(G)\to [k]\text{ s.t. $c$ is a $k$-coloring of $G$ and $c(v)\not\in r(v)$ for all $v\in V(G)$}}}.
\]
We drop the subscript $G$ when clear from context. This function counts the number of proper $k$-colorings in which no vertex is colored one of its colors forbidden by $r$. This function specializes to the usual chromatic polynomial for the trivial restraint $r_0$ with $r_0(v)=\emptyset$ for all $v\in V(G)$. The edge deletion-contraction recurrence for chromatic polynomials can be generalized to establish that $\rho_{G,r}(k)=p(k)$ for some polynomial $p$ and any $k\ge k_0$, where $k_0$ denotes the maximum integer that appears in $r$ \cite{Ere15}*{Theorem~4.1.2}.  We refer to these functions $\rho_{G,r}$ as polynomial for sufficiently large $k$. 

For graph $G$ and $k\in \mathbb{N}$, we can organize the $\pi_G(k)$ distinct $k$-colorings of $G$ in a graph $\mathcal C_k(G)$, called the \emph{$k$-coloring graph of $G$}, whose vertex set corresponds to the set of $k$-colorings of $G$ and whose edge set corresponds to pairs of $k$-colorings that differ on a single vertex of $G$. Figure~\ref{fig:C3P3} below depicts the 3-coloring graph for a 3-path, with each vertex of the coloring graph represented as a labeled rectangle.

\begin{figure}[ht]
\begin{tikzpicture}[baseline=0cm, scale=0.5]

% top square 
\node (a) at (0,0) {};
\node (b) at (-2,-2) {};
\node (c) at (0,-4) {};
\node (d) at (2,-2) {};
\node (e) at (-5,-6) {};
\node (f) at (5,-6) {};
\draw[style=thick] (b)--(e)--(-3,-8)--(-5,-10)--(5,-10)--(7,-8)--(5,-6)--(d);
\draw[style=thick] (e)--(-7,-8)--(-5,-10);
\draw[style=thick] (5,-6)--(3,-8)--(5,-10);
\draw[style=thick] (-5,-6)--(-3,-8)--(-5,-10);
\draw[style=thick] (a)--(b)--(c)--(d)--(a);
\draw[fill=white] (-1.5,-.5) rectangle (1.5,.5);
\draw[fill=white](-1,0)--(1,0);
\draw[radius=.35,fill=white](-1,0)circle node {1};
\draw[radius=.35,fill=white](0,0)circle node{2};
\draw[radius=.35,fill=white](1,0)circle node{3};

\draw[fill=white] (-3.5,-2.5) rectangle (-.5,-1.5);
\draw[style=thick](-3, -2)--(-1,-2);
\draw[radius=.35,fill=white](-3,-2)circle node{3};
\draw[radius=.35,fill=white](-2,-2)circle node{2};
\draw[radius=.35,fill=white](-1,-2)circle node{3};

\draw[fill=white] (3.5,-2.5) rectangle (.5,-1.5);
\draw[style=thick](3, -2)--(1,-2);
\draw[radius=.35,fill=white](3,-2)circle node{1};
\draw[radius=.35,fill=white](2,-2)circle node{2};
\draw[radius=.35,fill=white](1,-2)circle node{1};

\draw[fill=white] (-1.5,-4.5) rectangle (1.5,-3.5);
\draw[style=thick](-1,-4)--(1,-4);
\draw[radius=.35,fill=white](-1,-4)circle node{3};
\draw[radius=.35,fill=white](0,-4)circle node{2};
\draw[radius=.35,fill=white](1,-4)circle node{1}; 

% left square

\draw[fill=white] (-6.5,-6.5) rectangle (-3.5,-5.5);
\draw[style=thick](-6,-6)--(-4,-6);
\draw[radius=.35,fill=white](-6,-6)circle node{3};
\draw[radius=.35,fill=white](-5,-6)circle node{1};
\draw[radius=.35,fill=white](-4,-6)circle node{3};

\draw[fill=white] (-8.5,-8.5) rectangle (-5.5,-7.5);
\draw[style=thick](-8, -8)--(-6,-8);
\draw[radius=.35,fill=white](-8,-8)circle node{2};
\draw[radius=.35,fill=white](-7,-8)circle node{1};
\draw[radius=.35,fill=white](-6,-8)circle node{3};

\draw[fill=white] (-4.5,-8.5) rectangle (-1.5,-7.5);
\draw[style=thick](-2, -8)--(-4,-8);
\draw[radius=.35,fill=white](-2,-8)circle node{2};
\draw[radius=.35,fill=white](-3,-8)circle node{1};
\draw[radius=.35,fill=white](-4,-8)circle node{3};

\draw[fill=white] (-6.5,-10.5) rectangle (-3.5,-9.5);
\draw[style=thick](-6,-10)--(-4,-10);
\draw[radius=.35,fill=white](-6,-10)circle node{2};
\draw[radius=.35,fill=white](-5,-10)circle node{1};
\draw[radius=.35,fill=white](-4,-10)circle node{2}; 

% right square

\draw[fill=white] (3.5,-6.5) rectangle (6.5,-5.5);
\draw[style=thick](4,-6)--(6,-6);
\draw[radius=.35,fill=white](4,-6)circle node{1};
\draw[radius=.35,fill=white](5,-6)circle node{3};
\draw[radius=.35,fill=white](6,-6)circle node{1};

\draw[fill=white] (1.5,-8.5) rectangle (4.5,-7.5);
\draw[style=thick](2, -8)--(4,-8);
\draw[radius=.35,fill=white](2,-8)circle node{2};
\draw[radius=.35,fill=white](3,-8)circle node{3};
\draw[radius=.35,fill=white](4,-8)circle node{1};

\draw[fill=white] (5.5,-8.5) rectangle (8.5,-7.5);
\draw[style=thick](8, -8)--(6,-8);
\draw[radius=.35,fill=white](6,-8)circle node{1};
\draw[radius=.35,fill=white](7,-8)circle node{3};
\draw[radius=.35,fill=white](8,-8)circle node{2}; 

\draw[fill=white] (3.5,-10.5) rectangle (6.5,-9.5);
\draw[style=thick](4,-10)--(6,-10);
\draw[radius=.35,fill=white](4,-10)circle node{2};
\draw[radius=.35,fill=white](5,-10)circle node{3};
\draw[radius=.35,fill=white](6,-10)circle node{2}; 
\end{tikzpicture} 
\caption{$\mathcal C_3(P_3)$, the 3-coloring graph for $P_3$, with vertex labels indicating their underlying colorings}\label{fig:C3P3}  
\end{figure}
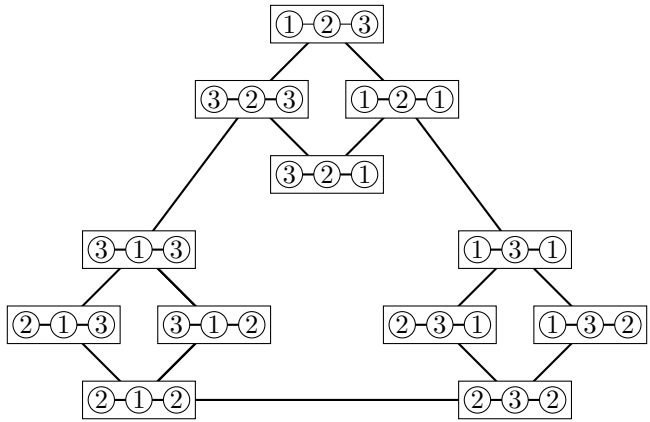

\section{Generalization of the Chromatic Polynomial}\label{sect:general}
In this section, we prove that for any fixed graphs $G$ and $H$, the number of times $H$ appears as an induced subgraph in $\mathcal C_k(G)$, the $k$-coloring graph of $G$, is polynomial in $k$ for sufficiently large $k$. We call this polynomial $\pi_G^{(H)}(k)$ the \emph{chromatic $H$-polynomial} of $G$. This generalizes the chromatic polynomial, with $\pi_G(k)=\pi_G^{(N_1)}(k)$. Our proof demonstrates a finite enumeration of all the ways to color particular subsets of $U\subseteq V(G)$ in order to yield an induced copy of $H$ in $\mathcal C_k(G)$, and it uses restrained chromatic polynomials to count the number of consistent colorings of $V(G-U)$.  We note that while deletion-contraction is used to prove polynomiality of the chromatic polynomial (and restrained chromatic polynomials), deletion-contraction does not hold for the chromatic $H$-polynomial of $G$ for general $H$. Instead, we rely on writing this function as a sum of restrained polynomials, each of which are known to be polynomial for sufficiently large $k$. We give the necessary notation and technical definitions before stating and proving the result. 

\subsection{Definitions} 

Fix graphs $G$ and $H$. Let $U$ be a subset of $V(G)$, and let $C$ be $\abs{V(H)}$ functions from the set of all functions $c\colon U\to \N$.
We call the pair $(U,C)$ an \emph{$H$-generator} in $G$ if the following two conditions hold: 
\begin{itemize}
    \item The elements of $C$ correspond to proper colorings of $G[U]$, in that for every $c\in C$ and $v_iv_{j}\in E(G[U])$, we have $c(v_i)\ne c(v_{j})$; and
    \item $H$ is isomorphic to the graph identified by vertex set $C$ with edges between every pair of colorings that differ on their assignment for a single vertex in $U$.
\end{itemize}
We further call $(U,C)$ a \emph{minimal $H$-generator} if the following two additional conditions hold:
\begin{itemize}
\item $U$ includes only vertices that take on at least two colors in $C$, that is, for every $v\in U$ there exist $c_1,c_2\in C$ with $c_1(v)\ne c_2(v)$; and 
    \item $C$ uses its full range of colors, that is, if $\kappa\in\N$ is in the image of some $c\in C$, then any value less than $\kappa$ must be in the image of some $c'\in C$. 
\end{itemize}

Intuitively, an $H$-generator describes a set of coloring changes in a subgraph of $G$ that corresponds to a realized induced copy of $H$ in the coloring graph of $G$.   
For example, consider $G=P_3$ (with vertices $v_1$, $v_2$, and $v_3$ in their natural, left-to-right ordering) and $H=P_3$.  Note that $G=H$ in this example, but $G$ corresponds to the base graph and $H$ is the subgraph we are counting in $\mathcal C_k(G)$. Let $U=\{v_1, v_3\}$ and $C=\{c_1, c_2, c_3\}$ where:
\begin{align}
\label{eq:c_example}
   & c_1(v_1)=2\ ,\ c_1(v_3) = 2 \nonumber\\& c_2(v_1)=2\ ,\ c_2(v_3) = 3 \\  &c_3(v_1)=3\ ,\ c_3(v_3) = 3 \nonumber\ .
\end{align} 
This $(U,C)$ is a $P_3$-generator (but not minimal as color 1 is not used), corresponding to at least one copy of $P_3$ in $\mathcal C_k(P_3)$ where $v_1$ and $v_3$ alternately change between colors 2 and 3.  There is only one such instance in $\mathcal C_3(P_3)$, which is indicated in the subgraph of Figure~\ref{fig:C3P3} reproduced in Figure~\ref{fig:C3P3-gen}. Other choices of $U$ and $C$ would correspond to additional copies of $P_3$ in $C_k(P_3)$.

\begin{figure}[ht]
\begin{tikzpicture}[baseline=0cm, scale=0.5]
% top square 
\draw[style=ultra thick] (-5,0)--(-7,-2)--(-5,-4);

% left square
\draw[fill=white] (-6.5,-.5) rectangle (-3.5,.5);
\draw[style=thick](-6,0)--(-4,0);
\draw[radius=.35,fill=white](-6,0)circle node{3};
\draw[radius=.35,fill=white](-5,0)circle node{1};
\draw[radius=.35,fill=white](-4,0)circle node{3};

\draw[fill=white] (-8.5,-2.5) rectangle (-5.5,-1.5);
\draw[style=thick](-8, -2)--(-6,-2);
\draw[radius=.35,fill=white](-8,-2)circle node{2};
\draw[radius=.35,fill=white](-7,-2)circle node{1};
\draw[radius=.35,fill=white](-6,-2)circle node{3};

\draw[fill=white] (-6.5,-4.5) rectangle (-3.5,-3.5);
\draw[style=thick](-6,-4)--(-4,-4);
\draw[radius=.35,fill=white](-6,-4)circle node{2};
\draw[radius=.35,fill=white](-5,-4)circle node{1};
\draw[radius=.35,fill=white](-4,-4)circle node{2}; 
\end{tikzpicture} 
\caption{The induced $P_3$ in $\mathcal C_3(P_3)$ generated by $(\set{v_1,v_3},C)$, for $C=\set{c_1,c_2,c_3}$ as in Eqs.~\eqref{eq:c_example}}\label{fig:C3P3-gen}  
\end{figure}
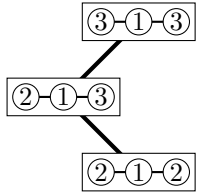

In counting the occurrences of $H$ in the coloring graphs for $G$, our result considers the restraints imposed by each $H$-generator. Specifically, for $H$-generator $(U,C)$, we specify a restraint on $G-U$ given by 
\begin{equation}r_{(U,C)}(v)=\set{j:c(u)=j\text{ for some $u\in N(v)\cap U$ and $c\in C$}} \text{ for $v\in V(G)-U$.}
\end{equation}
 This restraint corresponds to forbidding any vertex not in $U$ from having any of the colors assigned by $C$ to any neighbor of the vertex in $U$.

Returning to the example of $G=H=P_3$ with $U=\{v_1, v_3\}$, and $C=\{c_1, c_2, c_3\}$ as defined in Eq.~\eqref{eq:c_example}, in this case, $V(G)-U=\set{v_2}$ and $r_{(U,C)}$ is defined by $r_{(U,C)}(v_2)=\{2,3\}$.  In Figure \ref{fig:C3P3}, this corresponds to the copy of $P_3$ shown in Figure \ref{fig:C3P3-gen}, which occurs only where $v_2$ does not take colors 2 or 3. 

We say a minimal generator $(U,C)$ uses $\kappa$ colors if $\kappa$ is the largest value in the image of any $c\in C$. In the special case that $U=\emptyset$, we let $c_{\emptyset}$ denote the coloring with empty domain and say $(\emptyset,\set{c_\emptyset})$ uses zero colors, as in Corollary~\ref{cor:specializes} of our main result. Let $k_0(G,H)$ denote the largest value of $\kappa$ such that $G$ has a minimal $H$-generator that uses $\kappa$ colors, dropping the arguments $G$ and $H$ when clear from context.  This value $k_0$ is bounded above by $2|E(H)|$,  because every minimal generator must yield $|E(H)|$ edges that each correspond to a (not necessarily distinct) vertex of $G$ swapping between one of 2 colors. Then we let $\mathcal G^{(G,H)}$ denote the set of all minimal $H$-generators in $G$, dropping the superscript when clear from context, and we partition this set as $\mathcal G=\mathcal G_0\cup\mathcal G_1\cup \mathcal G_2\cup \cdots \cup \mathcal G_{k_0}$, where $\mathcal G_{\kappa}$ corresponds to minimal $H$-generators that use $\kappa$ colors. 

We also define a particular permutation that will be useful in establishing a bijection for our main result. For natural numbers $\kappa\le k$ and color palette $\mathcal P=\set{j_1,\dots,j_\kappa}$ with $j_1<\dots<j_\kappa\le k$, we let $\sigma_{\mathcal P}$ denote the permutation of $[k]$ defined such that 
\begin{align}
    &\sigma_{\mathcal P}(j_i)=i\text{ for }i\in[\kappa]\text{ , and } \notag \\
    &\sigma_{\mathcal P}(j)<\sigma_{\mathcal P}(j')\text{ for any $j<j'$ with $j,j'\not\in {\mathcal P}$.}\label{eq:palette-perm}
\end{align} This is the unique permutation that maps the elements of $\mathcal P$ to $[\kappa]$ while maintaining the order of colors in $\mathcal P$ and the order of colors not in $\mathcal P$.

\subsection{Polynomiality of Induced Subgraph Counts}
We now prove that the number of induced copies of a graph $H$ in $\mathcal C_k(G)$ grows as a polynomial function in $k$.  This polynomial function can be written as a sum of restrained chromatic polynomials whose restraints correspond to minimal $H$-generators.  The key idea of the proof is that every isomorphic copy of $H$ in $\mathcal C_k(G)$ corresponds to a unique $H$-generator by applying $\sigma_{\mathcal P}$ to the palette $\mathcal P$ used by that copy of $H$ in $\mathcal C_k(G)$.

\begin{theorem}\label{thm:general} Fix graphs $G$ and $H$, and let $\pi_{G}^{(H)}(k)$ denote the number of {induced} subgraphs of $\mathcal{C}_k(G)$ that are isomorphic to $H$. Then $\pi_{G}^{(H)}(k)$ is a polynomial function of $k$ for sufficiently large $k$. In particular,
\begin{equation}
\label{eq:thm}
\pi_G^{(H)}(k) = \sum_{\kappa=0}^{2|E(H)|}\sum_{(U,C)\in \mathcal G_{\kappa}} \binom{k}{\kappa}\cdot \rho_{r_{(U,C)}}(k) \ .
\end{equation}
\end{theorem}

\begin{proof}
We will show that for fixed graphs $G$ on $n$ vertices and $H$ on $\ell$ vertices and $m$ edges, there is a finite set of minimal $H$-generators that use at most $2m$ colors and none that use more.  
We argue each such generator contributes a number of occurrences of $H$ as an induced subgraph of $\mathcal C_k(G)$ that is polynomial in $k$ for $k\ge k_0$, and these contributions account for all occurrences of $H$ in $\mathcal C_k(G)$, establishing polynomiality of 
$\pi_G^{(H)}(k)$ for $k\ge k_0$.

To enumerate the minimal generators $(U,C)\in\mathcal G$, observe that there are {$\sum_{s=0}^m \binom{n}{s}$} candidate choices for $U\subseteq V(G)$ of size at most $m$, and for each $U$ there are at most $\binom{(2m)^m}{\ell}$ choices for $C$.  This is because each of $\ell$ functions must assign each of at most $m$ vertices to one of at most $2m$ colors. Hence, $\abs{\mathcal G}$ is finite.

Because each minimal generator uses only colors in $[k_0]$ (recalling $k_0\le 2m$), each $\rho_{r_{(U,C)}}(k)$ is guaranteed to be polynomial for $k\ge k_0$ \cite{Ere15}.  Therefore, it suffices to show that each $(U,C)\in \mathcal G_\kappa$ accounts for exactly $\binom{k}{\kappa}\cdot \rho_{r_{(U,C)}}(k)$ occurrences of $H$ in $\mathcal C_k(G)$ and that there are no other occurrences of $H$. 

For $k\ge k_0$, fix any particular occurrence of $H$ as an induced subgraph of $\mathcal C_k(G)$, and let $\mathcal P$ be the palette of $\kappa\le 2m$ colors used for the vertices of $G$ that change color in this induced subgraph. 
Applying $\sigma_{\mathcal{P}}$ to the colorings associated with the vertices of this subgraph yields a unique minimal generator and choice of coloring of the vertices of $G$ frozen in the occurrence of $H$ in $\mathcal C_k(G)$. 
Likewise, any choice of minimal generator $(U,C)$ combined with one of the $\rho_{r_{(U,C)}}(k)$ compatible assignments for $V(G)-U$ contributes $\binom{k}{\kappa}$ distinct occurrences of $H$ as an induced subgraph in $\mathcal C_k(G)$. 
Hence we have our desired expression for $k\ge k_0$. \end{proof}

\begin{remark}\label{remark:thmstatement}
Hogan, Scott, Tamitegama, and Tan~\cite{HSTT24} strengthened the statement of Theorem~\ref{thm:general} to hold for all $k\ge 1$, as opposed to sufficiently large $k$.  We note that our proof holds for $k\ge 1$ as well.  The restrained chromatic polynomials appearing in Eq.~\eqref{eq:thm} are all polynomial for $k\ge \kappa$, and the binomial term $\binom{k}{\kappa}$ is identically 0 for all $k < \kappa$.
\end{remark}

We note that the chromatic $N_1$-polynomial by definition is the chromatic polynomial. To check that the right-hand side of Eq.~\eqref{eq:thm} evaluates to $\pi_G(k)$, note that any minimal $N_1$-generator must have exactly $\abs{V(N_1)}=1$ coloring $c\in C$. Any $v\in U$ must take on at least two colors in $C$, so we have $U=\emptyset$ and $C$ consists of the unique coloring with empty domain, with this generator using zero colors.  {(Note in fact that $N_1$ is the {\em only} choice of $H$ with $\mathcal G_0\ne \emptyset$.)} Hence, the summation consists of this unique $(U,C)\in \mathcal G_0$. Then $\binom{k}{0}=1$, and because $U=\emptyset$, $r_{(U,C)}$ imposes no constraints, we have $\rho_{r_{(U,C)}}(k)=\pi_G(k)$.

\begin{corollary}\label{cor:specializes}
    For any graph $G$, the chromatic $N_1$-polynomial of $G$ equals the chromatic polynomial of $G$.
\end{corollary}

Before studying particular instantiations of chromatic $H$-polynomials, we make a few observations about these polynomials that apply to any $H$ and $G$ that do not rely on the polynomiality of these functions. In particular, we observe that $\pi_{G}^{(H)}$ is monotonic in three different senses.
\begin{enumerate}
    \itemsep 0.2cm
    \item \emph{(Monotonic in $k$)} Because $\mathcal C_k(G)$ is an induced subgraph of $\mathcal C_{k+1}(G)$, we have $\pi_G^{(H)}(k+1)\geq \pi_{G}^{(H)}(k)$ for each $k\in\mathbb{N}$. 
    \item \emph{(Monotonic in $H$)}  For any induced subgraph $H_1$ of graph $H_2$, every minimal $H_2$-generator in $G$ contains a minimal $H_1$-generator in $G$, so $\pi_{G}^{(H_1)}(k)\geq \pi_{G}^{(H_2)}(k)$ for each $k\in\mathbb{N}$.
    \item \emph{(Monotonic in $G$)} If $G_1$ is a subgraph of $G_2$ with the same vertex set, then $\mathcal C_k(G_2)$ is an induced subgraph of $ \mathcal C_k(G_1)$. Thus, $\pi_{G_1}^{(H)}(k)\geq \pi_{G_2}^{(H)}(k)$ for each $k\in\mathbb{N}$.
\end{enumerate}

\section{Instantiations}\label{sec:instantiations}

In this section, we illustrate how to instantiate new generalizations of this polynomial $\pi_G^{(H)}(k)$ beyond the well-studied chromatic polynomial with $H=N_1$. We begin with the smallest subgraph with an edge, $H=P_2$. The polynomial $\pi_G^{(P_2)}(k)$ counts the number of edges in the coloring graph $\mathcal C_k(G)$, which equivalently counts the number of pairs of $k$-colorings of $G$ that differ on exactly one vertex. For that reason, we refer to $\pi_G^{(P_2)}$ as the \emph{chromatic pairs polynomial}. 

We give a general formula in Eq.~\eqref{eq:edge_formula} for the chromatic pairs polynomial  as a sum of particular restrained chromatic polynomials, and we derive closed-form expressions for the chromatic pairs polynomial for certain families of graphs. 
We remark that these results immediately yield results for $H=N_2$, which counts the number of pairs of colorings that differ on {\em more} than a single vertex, because $\pi_G^{(N_2)}(k)=\binom{\pi_G^{(N_1)}(k)}{2}-\pi_G^{(P_2)}(k).$

After our discussion of chromatic pairs polynomials, we discuss how our general polynomial result can additionally be applied to count cliques and small cycles in coloring graphs.

\subsection{Counting Edges in Coloring Graphs}

To provide a general expression counting the number of pairs of colorings of a graph that differ on a single vertex, we first observe that minimal $P_2$-generators all have a particular form, regardless of the base graph $G$.  Any edge in the coloring graph corresponds to one vertex swapping between two colors, so minimality requires that any minimal $P_2$-generator $(U,C)$ is such that $U=\set{v}$, where $v$ is a single color-changing vertex, and $C=\set{c_1,c_2}$ with $c_1(v)=1,c_2(v)=2$.  Correspondingly let $r_{(\set{v},\set{c_1,c_2})}= r_{v}^{(2)}$ denote the restraint on $G-v$ that forbids neighbors of $v$ from taking colors 1 or 2. 
Then Theorem~\ref{thm:general} simplifies to the following formula counting edges in a coloring graph:
\begin{equation}
\label{eq:edge_formula}
\pi_G^{(P_2)}(k) = \binom{k}{2} \sum_{v\in V(G)} \rho_{r_{v}^{(2)}}(k) \ .
\end{equation}

In general, the forms of the individual $\rho_{r_v^{(2)}}$ are vertex- and graph-dependent.  We now provide a few specific examples for certain classes of graphs.

\subsubsection{Counting chromatic pairs for null and complete graphs.} The $\rho_{r_v^{(2)}}$ are particularly easy to compute for highly structured graphs such as null and complete graphs.  The vertices in a null graph have no neighbors, so when one vertex changes colors, the remaining $n-1$ vertices each have their full palette available.  Therefore, for any vertex $v$ in a null graph, $\rho_{r_v^{(2)}}(k)=\pi_{N_{n-1}}(k)=k^{n-1}$.  This yields the following chromatic pairs polynomial:
\begin{equation}\pi_{N_n}^{(P_2)}(k) = n \binom{k}{2} k^{n-1}  \ .
\end{equation}
On the other hand, any vertex of a complete graph restrains all other vertices from taking two colors.  Thus $\rho_{r_v^{(2)}}(k)=\pi_{K_{n-1}}(k-2)=(k-2)(k-3)\cdots(k-n)$.  The chromatic pairs polynomial is then given by:
\begin{equation}
\pi_{K_n}^{(P_2)}(k) = n \binom{k}{2} (k-2)\cdots (k-n) = \frac n 2 k(k-1)(k-2)\cdots (k-n) \ . 
\end{equation}

\subsubsection{Counting chromatic pairs for trees.} Null and complete graphs are vertex-transitive, meaning that for any two vertices there is an automorphism of the graph mapping one to another. This is not true for trees in general, so the restrained chromatic polynomials differ across vertices. However, we show that for any vertex $v$ in a tree $T$, $\rho_{r_v^{(2)}}$ depends only on the degree of $v$ in $T$. That is because each neighbor of $v$ has 2 colors forbidden by $v$ and then serves as the root of a subtree.  Each non-root node of these subtrees has exactly $k-1$ color options because it is restrained only by its parent. Hence, for any tree~$T$ with $\abs{V(T)}=n$ vertices, the number of edges in its $k$-coloring graph is given by
\begin{equation}
\label{eq:tree_formula}
    \pi_T^{(P_2)}(k) = \binom{k}{2} \sum_{v\in V(T)} (k-2)^{\deg(v)}(k-1)^{n-\deg(v)-1} \ .
\end{equation}

\subsubsection{Counting chromatic pairs for cycles and other pseudotrees.} We can similarly reason about the chromatic pairs polynomial for pseudotrees, which are connected graphs with equal numbers of vertices and edges. We first examine the special case of cycle graphs. 
We define the following restrained chromatic polynomials to help with the notation of the argument:
\begin{itemize}
\item $\sigma_n(k)$ is the number of $k$-colorings of $P_n$ with restraints $\set{1,2}$ for each leaf; and 
\item $\tau_n(k)$ is the number of $k$-colorings of $P_n$ with restraint $\set{1}$ for the first leaf and $\set{2}$ for the second.
\end{itemize}

Figure~\ref{fig:tau} illustrates how to write $\pi_{C_n}^{(P_2)}$ in terms of $\sigma_{n-1}$ and $\tau_n$ in terms of $\pi_{C_{n+2}}$, respectively. First, observe that each edge in $\mathcal C_k(C_n)$ corresponds to one of $n$ vertices swapping between any of $\binom{k}{2}$ pairs of colors.  The resulting number of ways to $k$-color the rest of the graph is exactly the number of ways to color $P_{n-1}$ such that neither leaf takes colors 1 or 2.  Hence, we can write the chromatic pairs polynomial for $C_n$ in terms of $\sigma_n$ as follows: 
\begin{equation}\label{eq:edgecycle}
\pi_{C_n}^{(P_2)}(k)=n \binom{k}{2} \sigma_{n-1}(k) \ .
\end{equation}
Next, note that we can color an $(n+2)$-cycle by fixing the colors of two adjacent vertices and coloring the remaining $P_n$ such that each leaf has a single distinct restraint, so we can write $\tau_n$ as follows:
\begin{equation}\label{eq:tau}
\tau_n(k)=\frac 1 {k(k-1)} \pi_{C_{n+2}}(k)\ .
\end{equation}

 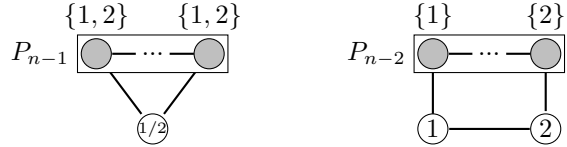
\begin{figure}[ht]
 \begin{tabular}{cc}
\begin{tikzpicture}[baseline=0cm, scale=0.5]
\node (a) at (0,0) {a};
\node (b) at (1.5,0) {b};
\node (c) at (3,0) {c};
\node (x) at (1.5,-2) {x};
\draw[style=thick] (a)--(x)--(c);
\draw[fill=white] (-.5,-.5) rectangle (3.5,.5);
\draw[style=thick] (a)--(c);
\draw[radius=.4,fill=black!25](0,0) circle;
\draw[radius=.4,fill=black!25] (3,0) circle;
\draw[radius=.4,fill=white] (1.5,-2) circle node{\tiny{1/2}};
\node at (0,1) {$\set{1,2}$};
\draw[radius=.45,fill=white,draw=none] (1.5,0) circle;
\node at (1.5,0) {...};
\node at (-1.5,0) {$P_{n-1}$};
%\draw[radius=.35,fill=black!25] (4.5,0) circle;
\node at (3,1) {$\set{1,2}$};
\end{tikzpicture} \qquad & \qquad 
\begin{tikzpicture}[baseline=0cm, scale=0.5]
\node (a) at (0,0) {a};
\node (b) at (1.5,0) {b};
\node (c) at (3,0) {c};
\node (x) at (0,-2) {x};
\node (y) at (3,-2) {7};
\draw[style=thick] (a)--(x)--(y)--(c);
\draw[fill=white] (-.5,-.5) rectangle (3.5,.5);
\draw[style=thick] (a)--(c);
\draw[radius=.4,fill=black!25](0,0) circle;
\draw[radius=.4,fill=black!25] (3,0) circle;
\draw[radius=.4,fill=white] (0,-2) circle node{1};
\draw[radius=.4,fill=white] (3,-2) circle node{2};
\node at (0,1) {$\set{1}$};
\draw[radius=.45,fill=white,draw=none] (1.5,0) circle;
\node at (1.5,0) {...};
\node at (-1.5,0) {$P_{n-2}$};
%\draw[radius=.35,fill=black!25] (4.5,0) circle;
\node at (3,1) {$\set{2}$};\end{tikzpicture}
\end{tabular}
\caption{Explaining $\sigma$ and $\tau$ in terms of cycles (Eqs.~\eqref{eq:edgecycle}~and~\eqref{eq:tau})}\label{fig:tau}
\end{figure}

Note that the chromatic polynomial for a cycle is $\pi_{C_n}(k)=(k-1)^n+(-1)^n(k-1)$ \cite{Rea68}*{Theorem~6}, yielding a closed-form expression for $\tau_n(k)$. 
We can further express $\sigma$ in terms of $\tau$ with the following inclusion-exclusion argument in order to give a closed-form expression for $\pi_{C_n}^{(P_2)}$. There are $k(k-1)^{n-1}$ unrestrained colorings of $P_n$.  Out of these colorings, there are $2(k-1)^{n-1}$ colorings that violate a restraint of $\sigma$ by coloring the left leaf 1 or 2, with another $2(k-1)^{n-1}$ colorings that violate the same right leaf restraint.  Subtracting these two sets from the total yields $(k-4)(k-1)^{n-1}$ colorings that satisfy the restraints $\{1,2\}$ for each leaf.  However, this undercounts the total number of such colorings because these two sets with violating colorings have a nonempty intersection.  We need to add back in the colorings in this overlap which include two types, depicted in Figure~\ref{fig:overlap}:

\begin{itemize}
    \item Colorings in which one leaf is 1 and the other is 2, and
    \item Colorings in which both leaves are 1 or both leaves are 2.
\end{itemize} 
\begin{figure}[ht]
\begin{tikzpicture}[baseline=0cm, scale=0.5]
\node (a) at (0,0) {a};
\node (b) at (1.5,0) {b};
\node (c) at (3,0) {c};
\node (d) at (4.5,0) {d};
\node (e) at (6,0) {e};
\node (f) at (7.5,0) {f};
\draw[fill=white] (1,-.5) rectangle (6.5,.5);
\draw[style=thick] (a)--(f);
\draw[radius=.4,fill=white](0,0) circle node{1}; 
\draw[radius=.4,fill=black!25] (1.5,0) circle;
\node at (1.5,1) {$\set{1}$};
\draw[radius=.45,fill=white,draw=none] (3.75,0) circle;
\node at (3.75,0) {...};
\node at (3.75,-1) {$P_{n-2}$};
%\draw[radius=.35,fill=black!25] (4.5,0) circle;
\node at (6,1) {$\set{2}$};
\draw[radius=.4,fill=black!25](6,0) circle;
\draw[radius=.4,fill=white](7.5,0) circle node{2};
\end{tikzpicture}
\qquad\qquad\begin{tikzpicture} [baseline=0cm, scale=0.45]
\node (a) at (0,0) {a};
\node (b) at (1.5,0) {b};
\node (c) at (3,0) {c};
\node (d) at (4.5,0) {d};
\node (e) at (6,0) {e};
\node (f) at (7.5,0) {f};
\draw[fill=white] (1,-.5) rectangle (5,.5);
\draw[style=thick] (a)--(f);
\draw[radius=.4,fill=white](0,0) circle node{1}; 
\draw[radius=.4,fill=black!25] (1.5,0) circle;
\node at (1.5,1) {$\set{1}$};
\draw[radius=.45,fill=white,draw=none] (3,0) circle;
\node at (3,0) {...};
\node at (3,-1) {$P_{n-3}$};
\draw[radius=.4,fill=black!25] (4.5,0) circle;
\node at (4.5,1) {$\set{c}$};
\draw[radius=.4,fill=white](6,0) circle node{$c$};
\draw[radius=.4,fill=white](7.5,0) circle node{1};
\end{tikzpicture}
\caption{Colorings of $P_n$ that violate $\set{1,2}$ restraints for both leaves}\label{fig:overlap}
\end{figure}
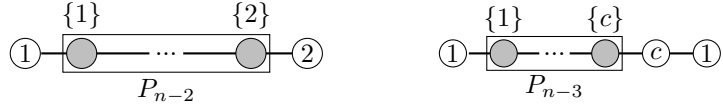

The number of colorings in which one leaf in $P_n$ is 1 and the other is 2 (left side of Figure~\ref{fig:overlap}) is the number of colorings of $P_{n-2}$ in which one leaf is not 1 and the other is not 2, which is exactly $2\tau_{n-2}(k)$, where the factor 2 accounts for which leaf is colored 1. 
The number of colorings in which both leaves are 1 (right side of Figure~\ref{fig:overlap}) is the number of colorings of $P_{n-3}$ in which one leaf is not 1, respecting the constraint of the left leaf of $P_n$, and the other leaf is not $c\ne 1$ for any of $k-1$ choices the neighbor of the right leaf in $P_n$. The number of colorings in which both leaves are 1 or both leaves are 2 is thus $2(k-1)\tau_{n-3}(k)$.  

Adding back these overlap colorings yields an expression for $\sigma$ in terms of $\tau$ as follows:
\[
\sigma_n(k) = (k-4)(k-1)^{n-1}+2(k-1)\tau_{n-3}(k)+2\tau_{n-2}(k) \ .
\]
Then substituting this into our closed-form expression for $\tau$ into Eq.~\eqref{eq:edgecycle}, we have the following chromatic pairs polynomial for cycle graphs:
\[
\pi_{C_n}^{(P_2)}(k)=\frac n 2 k(k-4)(k-1)^{n-1}+2n(k-1)^{n-1}+(-1)^n n(k-1)(k-2).
\]

We now use this result for cycle graphs to give a chromatic pairs polynomial for any pseudotree. For pseudotree $G$ on $n$ vertices with a unique cycle of length $\ell\le n$, let $d_1,\dots,d_\ell$ be the degrees of the cycle vertices $v_1,\dots,v_\ell$,  and let  $d_{\ell+1},\dots,d_n$ be the degrees of the remaining vertices $v_{\ell+1},\dots,v_n$, noting that $\sum_{i\in[n]}d_i=2n$. Every edge (occurrence of $H=P_2$) in $\mathcal C_k(G)$ corresponds to either a cycle vertex changing color or a non-cycle vertex changing color; we will count these types of edges separately. 

To count the occurrences of $P_2$ corresponding to cycle vertices, we note that of the $\pi_{C_{\ell}}^{(P_2)}(k)$ edges in an $\ell$-cycle's coloring graph, each vertex in $C_\ell$ contributes an equal proportion. In $\mathcal C_k(G)$, each of these edges due to cycle vertex $v_i$ appears $(k-2)^{d_i-2}(k-1)^{n-\ell-(d_i-2)}$ times because of the restraint this edge's colorings impose on the remaining $n-\ell$ non-cycle vertices. 

 Figure~\ref{fig:pseudocol} gives an example of how to count the occurrences of $P_2$ corresponding to a non-cycle vertex. In general, we first choose one of $\binom{k}{2}$ pairs of colors for non-cycle vertex $v_i$. Each of its neighbors has these two restraints, so there are $(k-2)^{d_i}$ neighborhood colorings. Then a breadth-first coloring of the rest of $v_i$'s tree gives $k-1$ options for each of those vertices, locking the cycle vertex present in that tree. With one cycle vertex color fixed, the rest of the cycle has $\pi_{C_\ell}(k)/k$ colorings, and then breadth-first coloring from the cycle vertices gives $k-1$ choices for each of the remaining vertices. 

\begin{figure}[ht]
 \begin{tikzpicture}[baseline=0cm, scale=0.5]
\node (v5) at (0,0) {};
\node (v4) at (2,0) {};
\node (v3) at (4,0) {};
\node (v2) at (4,2) {};
\node (v1) at (2,2) {};
\node (v6) at (6,2) {};
\node (v7) at (6,0) {};
\node (v8) at (8,2) {};
\node (v9) at (10,2) {};
\draw[style=thick] (v5)--(v3)--(v2)--(v1)--(v4);
\draw[style=thick] (v2)--(v9);
\draw[style=thick] (v7)--(v6);
{\footnotesize
\draw[radius=.4,fill=white](v5) circle;
\node[label=below:$k-1$] at (v5) {$v_5$};
\draw[radius=.4,fill=black!20](v8) circle;
\node[label=above:$\binom{k}{2}$] at (v8) {$v_8$};
\draw[radius=.4,fill=white](v6) circle;
\node[label=above:$k-2$] at (v6) {$v_6$};
\draw[radius=.4,fill=white](v9) circle;
\node[label=above:$k-2$] at (v9) {$v_9$};
\draw[radius=.4,fill=white](v4) circle node {$v_4$};
\node at (3,1) {$\frac{\pi_{C_\ell}(k)}{k}$};
\draw[radius=.4,fill=white](v1) circle node{$v_1$};
\draw[radius=.4,fill=white](v2) circle;
\node[label=above:$k-1$] at (v2) {$v_2$};
\draw[radius=.4,fill=white](v3) circle node {$v_3$};
\draw[radius=.4,fill=white] (v7) circle;
\node[label=below:$k-1$] at (v7) {$v_7$};}
\end{tikzpicture}
    \caption{Counting the $\binom{k}{2}\frac{\pi_{C_4}(k)}{k} (k-2)^2(k-1)^3$ occurrences of $P_2$ due to non-cycle vertex $v_8$}\label{fig:pseudocol}
\end{figure}
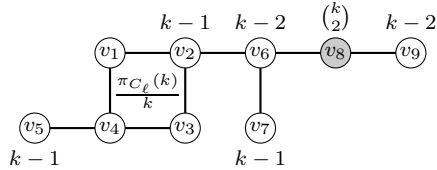

Counting each type of edge, we have the following chromatic pairs polynomial for pseudotree $G$:
\begin{equation}
\label{eq:ptree_formula}
\pi_G^{(P_2)}(k)=\frac{\pi_{C_\ell}^{(P_2)}(k)}{\ell}\sum_{i=1}^\ell (k-2)^{d_i-2}(k-1)^{n-\ell-(d_i-2)} + \binom{k}{2}\frac{\pi_{C_\ell}(k)}{k}\sum_{i=\ell+1}^n (k-2)^{d_i}(k-1)^{n-\ell-d_i}.
\end{equation}

\subsection{Counting Cliques in Coloring Graphs.} The chromatic pairs polynomial counts occurrences of $P_2$ in coloring graphs, which is the same as counting cliques of size 2. We can reason about any minimal $K_3$-generator $(U,C)$ in the same way that we reasoned that any minimal $P_2$-generator must be of a particular form. The coloring set $C$  for a $K_3$-generator must consist of three partial colorings. Furthermore, because $K_3$ has edges between each pair of its vertices, these colorings each must differ on a single vertex, so there can only be one color-changing vertex in $U$. Extending this argument, for any $t\ge 2$, any minimal $K_t$-generator must be of the form $(\set{v},\set{c_1,\dots,c_t})$, with $c_i(v)=i$ for $i\in[t]$. Then let $r_{(\set{v},\set{c_1,\dots,c_t})}=r_v^{(t)}$ denote the restraint on $G-v$ that forbids neighbors of $v$ from taking colors 1 through $t$, and we have the following generalization of Eq.~\eqref{eq:edge_formula} for counting chromatic $t$-tuples:
\begin{equation}
\label{eq:clique_formula}
\pi_G^{(K_t)}(k) = \binom{k}{t} \sum_{v\in V(G)} \rho_{r_{v}^{(t)}}(k).
\end{equation}
This formula is easily specialized for several classes of graphs we've just studied with $t=2$. Specifically, for null graphs, complete graphs, and trees we have
\begin{align}
\pi_{N_n}^{(K_t)}(k) &= n \binom{k}{t} k^{n-1} \ ,\label{eq:clique1} \\ 
\pi_{K_n}^{(K_t)}(k) &= n\binom{k}{t} (k-t)(k-t-1)\cdots (k-t-(n-2)) \ ,\label{eq:clique2} \\ 
\pi_{T}^{(K_t)}(k) &= \binom{k}{t} \sum_{v\in V} (k-t)^{\deg(v)} (k-1)^{n-\deg(v)-1} \ .\label{eq:clique3}
\end{align}
In Eq.~\eqref{eq:clique3}, we note that counting chromatic $t$-tuples for a tree has a similar dependence on degree sequence independent of $t$. We return to the implication of this in the following section. For now, rather than extending this analysis to counting cliques in coloring graphs for pseudotrees, we remark that  counting larger cliques may not provide much structural insight beyond counting $P_2$. Instead, we turn to exploring formulas for counting induced cycles in coloring graphs.

\subsection{Counting Cycles in Coloring Graphs} In this subsection we provide several additional examples of how to enumerate generators for different $H$ by focusing on small cycles. We will see that the structure of the restraints vary substantially depending on the length of the cycles we are searching for in the coloring graph. 
 
\subsubsection{Counting triangles.} 
Because a 3-cycle (triangle) is the same as  a 3-clique, Eqs.~\eqref{eq:clique_formula}-\eqref{eq:clique3} with $t=3$ show how to count triangles general graphs, null graphs, complete graphs, and trees, respectively. 

\subsubsection{Counting squares.} A more complex general expression can be used to calculate the number of induced 4-cycles (or squares) in a coloring graph. Any $C_4$-generator must involve two distinct vertices in $G$ changing color. This is because each vertex involved in a cycle in a coloring graph must be associated with at least two edges (so it can return to its original color), and if the same vertex were associated with all the cycle edges, then every pair of cycle vertices should be adjacent, creating a clique rather than an induced cycle.

If the two color-changing vertices in a $C_4$-generator are adjacent in $G$, they must take on two disjoint pairs of colors in $H$; if they are independent, their colors in the $C$ that induces $H$ may be disjoint or overlapping. For any $u,v\in V(G)$, we define the following restraints on $G-\set{u,v}$.

\begin{itemize}
\item $r_{uv}^{(4)}$ restrains neighbors of $u$ with $\set{1,2}$ and neighbors of $v$ with $\set{3,4}$; % and restrains no other vertices;
\item $r_{uv}^{(3)}$ restrains neighbors of $u$ with $\set{1,2}$ and neighbors of $v$ with $\set{1,3}$; and % and restrains no other vertices; and
\item $r_{uv}^{(2)}$ restrains neighbors of $u$ or $v$ with $\set{1,2}$.% and restrains no other vertices;
\end{itemize}
For $uv\in E(G)$, there are no minimal generators that use fewer than four colors. For $uv\not\in E(G)$, there is exactly one minimal generator that uses two colors. For $uv\not\in E(G)$, there are six minimal generators in $\mathcal{G}_3$ that use three colors. That is because any of the colors 1, 2, or 3 may be the one that colors both $u$ and $v$ in one of the vertices of the induced $H$, and either of the other two can be the other color assigned to $u$ in two other vertices of $H$. Similarly, for any $u,v\in V(G)$ there are six minimal generators in $\mathcal{G}_4$ that use four colors, because there are $\binom{4}{2}=6$ choices of two colors from $\set{1,2,3,4}$ that $u$ may take on in $H$. Hence, we can write the number of induced squares in a coloring graph as follows:

\begin{equation}\label{eq:C4}
\pi^{(C_4)}_G(k) = \sum_{\set{u,v}\in \binom{V(G)}{2}} 6\binom{k}{4} \rho_{r_{uv}^{(4)}}(k)  
+ \sum_{\substack{\set{u,v}\in \binom{V(G)}{2},\\uv\not\in E(G)}} 6 \binom{k}{3} \rho_{r_{uv}^{(3)}}(k) 
+ \sum_{\substack{\set{u,v}\in \binom{V(G)}{2},\\uv\not\in E(G)}} \binom{k}{2} \rho_{r_{uv}^{(2)}}(k) \ . 
\end{equation}
Note that this formula yields $\mathcal \pi_{P_3}^{(C_4)}(3)=3$, corresponding to the three squares that can be observed in Figure~\ref{fig:C3P3}. That is because the $\binom{k}{4}$ term goes to 0 with $k=3$ and because $\rho_{r_{uv}}^{(3)}(3)=0$ for outer (non-adjacent) vertices $u,v$, because the collective restraints on the center vertex would leave no remaining color for it. Finally we note that $\rho_{r_{uv}^{(2)}}(3)=1$ for outer vertices $u,v$ because their overlapping colors leave exactly one remaining color for the center vertex. That is why each of the squares in the figure have a fixed color for the center vertex. The $\binom{3}{2}=3$ choices of pairs of outer vertex colors yield the three squares.

\subsubsection{Counting five-cycles.} 

Coloring graphs never contain induced $5$-cycles, and so $\pi_{G}^{(C_5)}(k)=0$ \cite{BFHRS16}*{Corollary~12}. Subsection~\ref{sec:possible-cycles} summarizes this argument and explains why $C_5$ is the only cycle that cannot appear in coloring graphs.

Before continuing to the task of counting 6-cycles, we briefly note that 5-cycles are not the only $H$ for which the $H$-polynomial of any graph is identically zero. For example, it has been shown that any occurrence of $K_4-e$, the graph on four vertices with five edges, cannot be induced in a coloring graph. This is because any triangle in a coloring graph corresponds to a single vertex taking on three colors. Thus, any two triangles sharing an edge (see Figure~\ref{fig:Forbidden Theta}) must correspond to a single vertex taking on four colors without any other color changes, forcing the missing edge to appear in the coloring graph. In fact, there are infinitely many minimal forbidden subgraphs and therefore infinitely many $H$ such that $\pi_G^{(H)}(k)$ vanishes on all $G$ \cite{BFHRS16}*{Theorem~17} \cite{ABFR18}*{Theorem~13}.

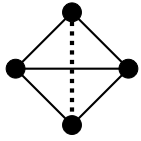
\begin{figure}[ht]
\begin{tikzpicture}[baseline=0cm, scale=.5]
\draw[thick] (0,1.5)--(1.5,0)--(0,-1.5)--(-1.5,0)--(0,1.5);
\draw[thick] (-1.5,0)--(1.5,0);
\draw[ultra thick, dotted] (0,1.5)--(0,-1.5);% to[out=180, in=90] (-1.75,0) to[out=270,in=180](0,-1);
\draw[radius=.25,fill=black] (0,1.5) circle;
\draw[radius=.25,fill=black] (0,-1.5) circle;
\draw[radius=.25,fill=black] (1.5,0) circle;
\draw[radius=.25,fill=black] (-1.5,0) circle;
\end{tikzpicture}
\caption{$K_4-e$ in a coloring graph must include the dotted line, so $\pi_G^{(K_4-e)}(k)=0$ for any $G$}
\label{fig:Forbidden Theta}   
\end{figure}

\subsubsection{Counting six-cycles.}\label{sec:6cycles} Generators for 6-cycles in coloring graphs fall into two  categories.
$C_6$ can appear as an induced subgraph of a coloring graph by two vertices alternately swapping between three colors each or by three vertices independently swapping between two colors each, the former of which may or may not occur as part of $K_3\square K_3$, and the latter of which always occurs as part of a 3-cube. For example, the thick edges in the graphs of Figure~\ref{fig:C6} depict an induced 6-cycle in the 3-coloring graph of $N_2$, which is $K_3\square K_3$, the 3-coloring graph of $P_2$, which itself is a 6-cycle, and an induced 6-cycle in the 2-coloring graph of $N_3$, which is a 3-cube.

\begin{figure}[ht]
\begin{tikzpicture}[baseline=0cm, scale=0.5]
\node (a) at (3,0) {a};
\node (b) at (6,0) {b};
\node (c) at (6,3) {c};
\node (d) at (0,3) {d};
\node (e) at (0,6) {e};
\node (f) at (3,6) {f};
\draw[style=thin] (3,6)--(6,6)--(6,3)--(0,3)--(0,0)--(3,0)--(3,6);
\draw[style=thin] (0,6)--(0,0);
\draw[style=thin] (0,6) to[out=225,in=135] (0,0);
\draw[style=ultra thick] (3,6) to[out=225,in=135] (3,0);
\draw[style=thin] (6,6) to[out=225,in=135] (6,0);
\draw[style=thin] (0,6) to[out=45, in=135] (6,6);
\draw[style=ultra thick] (0,3) to[out=45, in=135] (6,3);
\draw[style=thin] (0,0) to[out=45, in=135] (6,0);
\draw[style=ultra thick] (a)--(b)--(c);
\draw[style=ultra thick] (d)--(e)--(f);
\draw[fill=white] (2,-.5) rectangle (4,.5);
%\draw[style=thick](2.5,0)--(3.5,0);
\draw[radius=.35,fill=white](2.5,0) circle node{2};
\draw[radius=.35,fill=white](3.5,0) circle node{1};
\draw[fill=white] (5,-.5) rectangle (7,.5);
%\draw[style=thick](5.5,0)--(6.5,0);
\draw[radius=.35,fill=white](5.5,0) circle node{3};
\draw[radius=.35,fill=white](6.5,0) circle node{1};
\draw[fill=white] (-1,-.5) rectangle (1,.5);
%\draw[style=thick](2.5,0)--(3.5,0);
\draw[radius=.35,fill=white](-.5,0) circle node{1};
\draw[radius=.35,fill=white](.5,0) circle node{1};
\draw[fill=white] (2,2.5) rectangle (4,3.5);
%\draw[style=thick](2.5,0)--(3.5,0);
\draw[radius=.35,fill=white](2.5,3) circle node{2};
\draw[radius=.35,fill=white](3.5,3) circle node{2};
\draw[fill=white] (5,5.5) rectangle (7,6.5);
%\draw[style=thick](2.5,0)--(3.5,0);
\draw[radius=.35,fill=white](5.5,6) circle node{3};
\draw[radius=.35,fill=white](6.5,6) circle node{3};
\draw[fill=white] (5,2.5) rectangle (7,3.5);
%\draw[style=thick](5.5,3)--(6.5,3);
\draw[radius=.35,fill=white](5.5,3) circle node{3};
\draw[radius=.35,fill=white](6.5,3) circle node{2};
\draw[fill=white] (-1,2.5) rectangle (1,3.5);
%\draw[style=thick](-.5,3)--(.5,3);
\draw[radius=.35,fill=white](-.5,3) circle node{1};
\draw[radius=.35,fill=white](.5,3) circle node{2};
\draw[fill=white] (-1,5.5) rectangle (1,6.5);
%\draw[style=thick](-.5,6)--(.5,6);
\draw[radius=.35,fill=white](-.5,6) circle node{1};
\draw[radius=.35,fill=white](.5,6) circle node{3};
\draw[fill=white] (2,5.5) rectangle (4,6.5);
%\draw[style=thick](2.5,6)--(3.5,6);
\draw[radius=.35,fill=white](2.5,6) circle node{2};
\draw[radius=.35,fill=white](3.5,6) circle node{3};
\end{tikzpicture}
\qquad\qquad
\begin{tikzpicture}[baseline=0cm, scale=0.5]
\node (a) at (3,0) {a};
\node (b) at (6,0) {b};
\node (c) at (6,3) {c};
\node (d) at (0,3) {d};
\node (e) at (0,6) {e};
\node (f) at (3,6) {f};
\draw[style=ultra thick] (a)--(b)--(c)--(d)--(e)--(f)--(a);
\draw[fill=white] (2,-.5) rectangle (4,.5);
\draw[style=thick](2.5,0)--(3.5,0);
\draw[radius=.35,fill=white](2.5,0) circle node{2};
\draw[radius=.35,fill=white](3.5,0) circle node{1};
\draw[fill=white] (5,-.5) rectangle (7,.5);
\draw[style=thick](5.5,0)--(6.5,0);
\draw[radius=.35,fill=white](5.5,0) circle node{3};
\draw[radius=.35,fill=white](6.5,0) circle node{1};
\draw[fill=white] (5,2.5) rectangle (7,3.5);
\draw[style=thick](5.5,3)--(6.5,3);
\draw[radius=.35,fill=white](5.5,3) circle node{3};
\draw[radius=.35,fill=white](6.5,3) circle node{2};
\draw[fill=white] (-1,2.5) rectangle (1,3.5);
\draw[style=thick](-.5,3)--(.5,3);
\draw[radius=.35,fill=white](-.5,3) circle node{1};
\draw[radius=.35,fill=white](.5,3) circle node{2};
\draw[fill=white] (-1,5.5) rectangle (1,6.5);
\draw[style=thick](-.5,6)--(.5,6);
\draw[radius=.35,fill=white](-.5,6) circle node{1};
\draw[radius=.35,fill=white](.5,6) circle node{3};
\draw[fill=white] (2,5.5) rectangle (4,6.5);
\draw[style=thick](2.5,6)--(3.5,6);
\draw[radius=.35,fill=white](2.5,6) circle node{2};
\draw[radius=.35,fill=white](3.5,6) circle node{3};    
\end{tikzpicture}
\qquad\qquad
\begin{tikzpicture} [baseline=0cm, scale=0.5]
% top square 
\node (a) at (0,0) {a};
\node (b) at (4,0) {b};
\node (c) at (4,4) {c};
\node (d) at (0,4) {d};
\node (e) at (2,1.5) {e};
\node (f) at (6,1.5) {f};
\node (g) at (6,5.5) {g};
\node (h) at (2,5.5) {h};
\draw[style=thin] (a)--(b)--(c)--(d)--(a)--(e)--(f)--(g)--(h)--(e);
\draw[style=thin] (d)--(h)--(g)--(c)--(b)--(f)--(e)--(a);
\draw[style=ultra thick] (a)--(b)--(c)--(g)--(h)--(e)--(a);
\draw[fill=white] (-1.5,-.5) rectangle (1.5,.5);%a
\draw[radius=.35,fill=white](-1,0)circle node {1};
\draw[radius=.35,fill=white](0,0)circle node{1};
\draw[radius=.35,fill=white](1,0)circle node{1};
\draw[fill=white] (2.5,-.5) rectangle (5.5,.5);%b
\draw[radius=.35,fill=white](3,0)circle node {2};
\draw[radius=.35,fill=white](4,0)circle node{1};
\draw[radius=.35,fill=white](5,0)circle node{1};
\draw[fill=white] (-1.5,3.5) rectangle (1.5,4.5);%d
\draw[radius=.35,fill=white] (-1,4)circle node{1};
\draw[radius=.35,fill=white] (0,4)circle node{2};
\draw[radius=.35,fill=white] (1,4) circle node{1};
\draw[fill=white] (2.5,3.5) rectangle (5.5,4.5);%c
\draw[radius=.35,fill=white] (3,4) circle node{2};
\draw[radius=.35,fill=white] (4,4) circle node{2};
\draw[radius=.35,fill=white] (5,4) circle node{1};
\draw[fill=white] (.5,1) rectangle (3.5,2);%a
\draw[radius=.35,fill=white](1,1.5)circle node {1};
\draw[radius=.35,fill=white](2,1.5)circle node{1};
\draw[radius=.35,fill=white](3,1.5)circle node{2};
\draw[fill=white] (4.5,1) rectangle (7.5,2);%b
\draw[radius=.35,fill=white](5,1.5)circle node {2};
\draw[radius=.35,fill=white](6,1.5)circle node{1};
\draw[radius=.35,fill=white](7,1.5)circle node{2};
\draw[fill=white] (.5,5) rectangle (3.5,6);%d
\draw[radius=.35,fill=white] (1,5.5)circle node{1};
\draw[radius=.35,fill=white] (2,5.5)circle node{2};
\draw[radius=.35,fill=white] (3,5.5) circle node{2};
\draw[fill=white] (4.5,5) rectangle (7.5,6);%c
\draw[radius=.35,fill=white] (5,5.5) circle node{2};
\draw[radius=.35,fill=white] (6,5.5) circle node{2};
\draw[radius=.35,fill=white] (7,5.5) circle node{2};

\end{tikzpicture}
    \caption{Induced 6-cycles in $\mathcal C_3(N_2)$, $\mathcal C_3(P_2)$, and $\mathcal C_2(N_3)$}
    \label{fig:C6}
\end{figure}
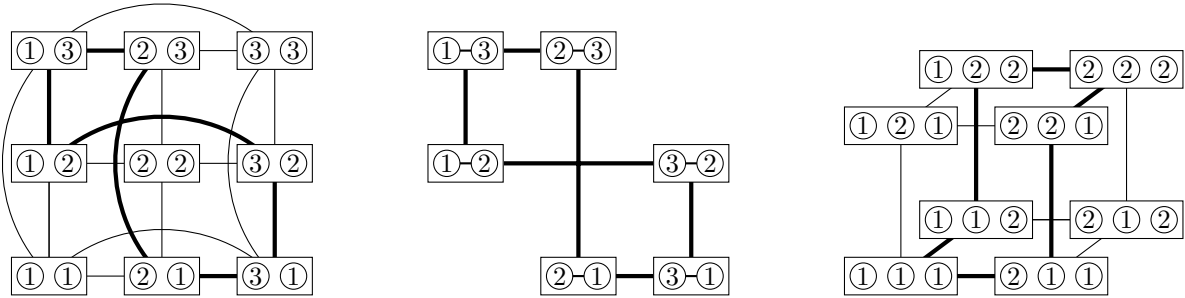

The two-vertex $C_6$-generators use at least three and at most six colors. The three-vertex $C_6$-generators use at least two colors (which is only possible if the three-vertices are pairwise non-adjacent) and up to six. Moreover, there are multiple three-vertex $C_6$-generators that use three or four colors. Writing down a closed-form formula for $\pi_{G}^{(C_6)}$ is substantially more complicated than our formula for $\pi_G^{(C_4)}$ albeit mechanically similar. We do this for completeness in Appendix~\ref{app:6cycles}, with an expression given in Eq.~\ref{eq:C6}.

\subsubsection{General approach for counting cycles}\label{sec:possible-cycles}
We conclude this section by describing the framework used for counting cycles in the previous examples (cf. \cite{ABFR18}). Each induced $\ell$-cycle on colorings has an associated partition $\lambda = (\lambda_1,\ldots, \lambda_m)\vdash \ell$ with $\lambda_1\geq\cdots \lambda_m>0$ where $m$ is the number of distinct vertices of the base graph that change color as the cycle is traversed, and $\lambda_i$ is the number of times the $i^{th}$ vertex changes color. 

A partition describing an induced $\ell$-cycle in this way must satisfy certain conditions. First, the partition must have  $\lambda_i>1$ for all $i$ since a cycle starts and ends with the same coloring.  For $\ell>3$, the same vertex cannot change color twice consecutively since this would imply existence of a chord in the cycle. Hence a partition coming from an induced $\ell$-cycle with $\ell>3$ must have $\lambda_i \leq \lfloor\frac{\ell}{2} \rfloor$ for all $i$. Any partition of $\ell$ not meeting these criteria can be disregarded when counting induced $\ell$-cycles on the colorings of a chosen base graph. Because no partition of $\ell=5$ satisfies these conditions, we can conclude $\pi_G^{(C_5)}(k)=0$ for all $G$. 

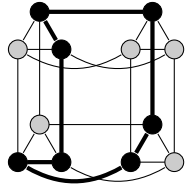
\begin{figure}[ht]
\begin{tikzpicture}[baseline=0cm, scale=0.5]
\node (a11) at (0,1) {};
\node (b11) at (.58,0) {};
\node (c11) at (-.58,0) {};
\node (a12) at (0,-2) {};
\node (b12) at (.58,-3) {};
\node (c12) at (-.58,-3) {};
\node (a21) at (3,1) {};
\node (b21) at (3.58,0) {};
\node (c21) at (2.42,0) {};
\node (a22) at (3,-2) {};
\node (b22) at (3.58,-3) {};
\node (c22) at (2.42,-3) {};
\draw[style=thin] (a11)--(b11)--(c11)--(a11);
\draw[style=thin] (a12)--(b12)--(c12)--(a12);
\draw[style=thin] (a21)--(b21)--(c21)--(a21);
\draw[style=thin] (a22)--(b22)--(c22)--(a22);
%\draw[style=ultra thick] (a11)--(b11);
\draw[style=thin] (a11)--(a12)--(a22)--(a21)--(a11);
\draw[style=thin] (b11) to[bend right] (b21);
\draw[style=thin] (b22) to[bend left] (b12);
\draw[style=thin] (b11) -- (b12);
\draw[style=thin] (b21) -- (b22);
\draw[style=thin] (c11) to[bend right] (c21);
\draw[style=thin] (c22) to[bend left] (c12);
\draw[style=thin] (c11) -- (c12);
\draw[style=thin] (c21) -- (c22);
\draw[radius=.25,fill=black](a11) circle;
\draw[radius=.25,fill=black](b11) circle;
\draw[radius=.25,fill=black!20](c11) circle;
\draw[radius=.25,fill=black!20](a12) circle;
\draw[radius=.25,fill=black](b12) circle;
\draw[radius=.25,fill=black](c12) circle;
\draw[radius=.25,fill=black](a21) circle;
\draw[radius=.25,fill=black!20](b21) circle;
\draw[radius=.25,fill=black!20](c21) circle;
\draw[radius=.25,fill=black](a22) circle;
\draw[radius=.25,fill=black!20](b22) circle;
\draw[radius=.25,fill=black](c22) circle;
\draw[style=ultra thick] (a11)--(b11)--(b12)--(c12);
\draw[style=ultra thick] (c12) to[bend right] (c22);
\draw[style=ultra thick] (c22)--(a22)--(a21)--(a11);
\end{tikzpicture}
    \caption{An induced 7-cycle in $K_3\ \square\ Q_2$}
    \label{fig:C7}
\end{figure}

For all other $\ell>2$, induced $\ell$-cycles in coloring graphs exist \cite{BFHRS16}*{Corollary~12}. For example, Figure~\ref{fig:C7} depicts an induced 7-cycle that arises in $K_3\ \square\ Q_2$, an induced subgraph of $\mathcal C_3(N_3)$. Of course, establishing the existence of a cycle in a coloring graph does not count the number of cycles. As illustrated in the above examples, enumerating the types of minimal $C_\ell$-generators in a graph becomes increasingly unwieldy as $\ell$ gets large.

\section{Invariants}\label{sect:invariants}
One of our main motivations in defining and studying the polynomial $\pi_G^{(H)}$ is its use as a refined graph invariant.  We have already observed that the choice of $H=N_1$ recovers the chromatic polynomial, so the $H$-polynomial generalizes the chromatic polynomial.  While the chromatic polynomial is a known invariant for many graphs, there are many (common) graphs that share a chromatic polynomial.  This is the case for all trees on $n$ vertices.  Even the well-known Tutte polynomial, which also generalizes the chromatic polynomial, does not distinguish any trees with the same number of vertices.  

We begin this section with a proof that the chromatic pairs polynomial $\pi_G^{(P_2)}$ can indeed distinguish certain trees on $n$ vertices.  However, this refined invariant is not a complete invariant for trees as it cannot distinguish non-isomorphic trees with the same degree sequence.  The remainder of the section explores structural properties of the chromatic pairs polynomial that can be used to distinguish other graphs, as well as a brief exploration of other choices of $H$ for which $\pi_G^{(H)}$ that can be used as additional invariants.

\subsection{Chromatic Pairs Polynomials as Refined Invariants for Trees}
\label{sec:distinguish_trees}
 The chromatic polynomial for any tree $T$ on $n$ vertices is $\pi_T(k)=\pi_T^{(N_1)}(k)=k(k-1)^n$, which is only determined by the number of vertices in the tree. 
In contrast, our closed-form expression for the chromatic pairs polynomial for a tree can be used as a refined but not complete invariant for trees. Figure~\ref{fig:trees} gives three 6-vertex trees, which have matching chromatic polynomials; the first has a distinct degree sequence and a distinct chromatic pairs polynomial and the second two have matching degree sequences and matching chromatic pairs polynomials. Theorem~\ref{thm:treedeg} establishes that the chromatic pairs polynomial for trees precisely distinguish trees with different degree sequences.

\tikzset{
  P4/.pic = {
  \draw[style=thick](0,0)--(3,0);
\draw[radius=.25,fill=black](0,0)circle;
\draw[radius=.25,fill=black](1,0)circle;
\draw[radius=.25,fill=black](2,0)circle;
\draw[radius=.25,fill=black](3,0)circle;
  },
}

\tikzset{
  K13/.pic = {
  \draw[style=thick](0,0)--(2,0)--(1,0)--(1,1);
\draw[radius=.25,fill=black](0,0)circle;
\draw[radius=.25,fill=black](2,0)circle;
\draw[radius=.25,fill=black](1,0)circle;
\draw[radius=.25,fill=black](1,1)circle;
  },
}

\tikzset{
  P6/.pic = {
  \draw[style=thick](0,0)--(5,0);
\draw[radius=.25,fill=black](0,0)circle;
\draw[radius=.25,fill=black](1,0)circle;
\draw[radius=.25,fill=black](2,0)circle;
\draw[radius=.25,fill=black](3,0)circle;
\draw[radius=.25,fill=black](4,0)circle;
\draw[radius=.25,fill=black](5,0)circle;
  },
}

\tikzset{
  T6a/.pic = {
  \draw[style=thick](0,0)--(4,0)--(2,0)--(2,1);
\draw[radius=.25,fill=black](0,0)circle;
\draw[radius=.25,fill=black](1,0)circle;
\draw[radius=.25,fill=black](2,0)circle;
\draw[radius=.25,fill=black](2,1)circle;
\draw[radius=.25,fill=black](3,0)circle;
\draw[radius=.25,fill=black](4,0)circle;
  },
}

\tikzset{
  T6b/.pic = {
  \draw[style=thick](0,0)--(4,0)--(3,0)--(3,1);
\draw[radius=.25,fill=black](0,0)circle;
\draw[radius=.25,fill=black](1,0)circle;
\draw[radius=.25,fill=black](2,0)circle;
\draw[radius=.25,fill=black](3,1)circle;
\draw[radius=.25,fill=black](3,0)circle;
\draw[radius=.25,fill=black](4,0)circle;
  },
}
 \tikzset{
  G5b/.pic = {
  \draw[style=thick](2,0)--(1,-1)--(0,0)--(1,1)--(2,0)--(3,0);
\draw[radius=.25,fill=black](0,0)circle;
\draw[radius=.25,fill=black](1,1)circle;
\draw[radius=.25,fill=black](2,0)circle;
\draw[radius=.25,fill=black](1,-1)circle;
\draw[radius=.25,fill=black](3,0)circle;
  },
} 

 \tikzset{
  G5a/.pic = {
  \draw[style=thick](1,0)--(0,-1)--(0,1)--(1,0)--(3,0);
\draw[radius=.25,fill=black](1,0)circle;
\draw[radius=.25,fill=black](0,-1)circle;
\draw[radius=.25,fill=black](0,1)circle;
\draw[radius=.25,fill=black](2,0)circle;
\draw[radius=.25,fill=black](3,0)circle;
  },
}

\tikzset{
G6a/.pic = {
  \draw[style=thick](1,0)--(0,-.5)--(0,.5)--(1,0)--(2,0)--(3,.5)--(2,0)--(3,-.5);
\draw[radius=.25,fill=black](1,0)circle;
\draw[radius=.25,fill=black](0,-.5)circle;
\draw[radius=.25,fill=black](0,.5)circle;
\draw[radius=.25,fill=black](2,0)circle;
\draw[radius=.25,fill=black](3,.5)circle;
\draw[radius=.25,fill=black](3,-.5)circle;
  },
}

\tikzset{
G6b/.pic = {
  \draw[style=thick](0,0)--(4,0)--(2,0)--(1.5,1)--(1,0);
\draw[radius=.25,fill=black](0,0)circle;
\draw[radius=.25,fill=black](1,0)circle;
\draw[radius=.25,fill=black](2,0)circle;
\draw[radius=.25,fill=black](3,0)circle;
\draw[radius=.25,fill=black](4,0)circle;
\draw[radius=.25,fill=black](1.5,1)circle;
  },
}

\tikzset{
G7a/.pic = {
\node (a) at (0,0) {};
\node (b) at (3,0) {};
\node (c) at (3,3) {};
\node (d) at (0,3) {};
\node (e) at (1.5,0.5) {};
\node (f) at (1.5,1.5) {};
\node (g) at (1.5,2.5) {};
\draw[style=thick] (a)--(b)--(c)--(d)--(a)--(e)--(f)--(g);
\draw[style=thick] (a)--(g)--(c);
\draw[style=thick] (b)--(g)--(d);
\draw[style=thick] (a)--(f);
\draw[radius=.25,fill=black](a) circle;
\draw[radius=.25,fill=black](b) circle;
\draw[radius=.25,fill=black](c) circle;
\draw[radius=.25,fill=black](d) circle;
\draw[radius=.25,fill=black](e) circle;
\draw[radius=.25,fill=black](f) circle;
\draw[radius=.25,fill=black](g) circle;
}
}

\tikzset{
G7b/.pic = {
\node (a) at (0,0) {};
\node (b) at (3,0) {};
\node (c) at (3,3) {};
\node (d) at (0,3) {};
\node (e) at (1.5,0.5) {};
\node (f) at (2.5,1.5) {};
\node (g) at (1.5,1.5) {};
\draw[style=thick] (a)--(b)--(c)--(d)--(a)--(e)--(b)--(f)--(c);
\draw[style=thick] (a)--(g)--(c);
\draw[style=thick] (b)--(g)--(d);
\draw[radius=.25,fill=black](a) circle;
\draw[radius=.25,fill=black](b) circle;
\draw[radius=.25,fill=black](c) circle;
\draw[radius=.25,fill=black](d) circle;
\draw[radius=.25,fill=black](e) circle;
\draw[radius=.25,fill=black](f) circle;
\draw[radius=.25,fill=black](g) circle;
}
}

\begin{figure}[ht]
\begin{tabular}{lll}
\begin{tabular}{l}
     $T_1=$\end{tabular} &
     \begin{tabular}{l}\tikz \node [scale=0.5, inner sep=0] {
   \begin{tikzpicture}
     \pic at (0,0) {P6};
   \end{tikzpicture}
}; \end{tabular} & 
\begin{tabular}{l} $\pi_{T_1}^{(P_2)}(k)=3k^7-23k^6+72k^5-118k^4+107k^3-51k^2+10k$
\end{tabular}\\  && \\  
\begin{tabular}{l}$T_2=$\end{tabular} &
\begin{tabular}{l}\tikz \node [scale=.5, inner sep=0] {
\begin{tikzpicture}
    \pic at (0,0) {T6a};
\end{tikzpicture}
}; \end{tabular}& 
\begin{tabular}{l}$\pi_{T_2}^{(P_2)}(k)=3k^7-23k^6+\frac{145}{2}k^5-\frac{241}{2}k^4+\frac{223}{2}k^3-\frac{109}{2}k^2+11k$ \end{tabular}\\ & & \\ 
\begin{tabular}{l}$T_3=$\end{tabular} & 
\begin{tabular}{l}\tikz\node [scale=.5,inner sep=0] {
\begin{tikzpicture}
    \pic at (0,0) {T6b};
\end{tikzpicture}
};\end{tabular} & 
\begin{tabular}{l} $\pi_{T_3}^{(P_2)}(k)=3k^7-23k^6+\frac{145}{2}k^5-\frac{241}{2}k^4+\frac{223}{2}k^3-\frac{109}{2}k^2+11k$\end{tabular}
\end{tabular}
\caption{Trees distinguished and not distinguished by their chromatic pairs polynomials}\label{fig:trees}
\end{figure}

\begin{theorem}\label{thm:treedeg}
    For trees $T_1$ and $T_2$, $\pi_{T_1}^{(P_2)}=\pi_{T_2}^{(P_2)}$ if and only if $T_1$ and $T_2$ have the same degree sequence.
\end{theorem}

\begin{proof}
    The reverse implication  is an immediate consequence of the closed form expression for the chromatic pairs polynomial for trees in Eq.~\eqref{eq:tree_formula} which depends entirely on the degree sequence of the tree. To see why the other direction holds, fix tree $T$ on $n$ vertices with max degree $\Delta$, and suppose $n_1,n_2,\dots,n_\Delta$ with $\sum_{d\in[\Delta]}n_d=n$ are the number of vertices in $T$ with degrees $1,2,\dots,\Delta$, respectively. Note that we can rewrite the chromatic pairs polynomial as follows:
    \[
    \pi_T^{(P_2)}(k)=\binom{k}{2}\sum_{d\in[\Delta]} n_d (k-2)^d(k-1)^{n-d-1}
    \]
    Then we can exactly recover the degree sequence $n_d$ for each $d=1,\dots,\Delta$ in turn as follows:
    \[
    n_d={\frac{d}{dk}\left(\frac{ \frac{\pi_T^{(P_2)}(k)}{\binom{k}{2}}
    -\sum_{i\in[d-1]} n_i(k-2)^i(k-1)^{n-i-1}}{(k-2)^{d-1}} \right)\bigg\rvert_{k=2} } \ .
    \]
    
    Therefore, if there is another tree that has this same chromatic pairs polynomial, it must have the same degree sequence.
\end{proof}

Although a tree's degree sequence uniquely determines its chromatic pairs polynomial and vice versa, this is not the case for graphs in general. Indeed, degree sequence does not even determine chromatic pairs polynomials for pseudotrees, trees with just one additional edge.  
Figure~\ref{fig:G6s} depicts two {pseudotrees} with the same degree sequence but different chromatic pairs polynomials. Additionally, a graph's chromatic pairs polynomial does not in general determine its degree sequence.  Figure~\ref{fig:G7s} provides two graphs with matching chromatic pairs polynomials but different degree sequences. 

\begin{figure}[ht]
\begin{tabular}{lll}
 \begin{tabular}{l}{$R_1=$}\end{tabular} &
 \begin{tabular}{l}\tikz \node [scale=0.5, inner sep=0] {
   \begin{tikzpicture}
     \pic at (0,0) {G6a};
   \end{tikzpicture}
}; \end{tabular}& 
\begin{tabular}{l}$\pi_{R_1}^{(P_2)}(k)=3k^7-27k^6+\frac{197}{2}k^5-187k^4+\frac{391}{2}k^3-107k^2+24k$\end{tabular} \\ & & \\
    \begin{tabular}{l}$R_2=$\end{tabular}&\begin{tabular}{l}\tikz \node [scale=0.5, inner sep=0] {
   \begin{tikzpicture}
     \pic at (0,0) {G6b};
   \end{tikzpicture}
};\end{tabular} &  \begin{tabular}{l}$\pi_{R_2}^{(P_2)}(k)=3k^7-27k^6+\frac{197}{2}k^5-\frac{373}{2}k^4+\frac{387}{2}k^3-\frac{209}{2}k^2+23k$ \end{tabular}
\end{tabular}
\caption{Graphs with the same degree sequence and different chromatic pairs polynomials}\label{fig:G6s}
\end{figure}

\begin{figure}[ht]
\begin{tabular}{lll}
 \begin{tabular}{l}$G_1=$\end{tabular}&\begin{tabular}{l}\tikz \node [scale=0.5, inner sep=0] {
   \begin{tikzpicture}
     \pic at (0,0) {G7a};
   \end{tikzpicture}
};\end{tabular} & \begin{tabular}{l}$\pi_{G_1}^{(P_2)}(k)=\frac{7}{2}k^8-\frac{115}{2}k^7+\frac{807}{2}k^6-\frac{3123}{2}k^5+3582k^4-4843k^3+3547k^2-1074k$ \end{tabular}\\ & & \\
    \begin{tabular}{l}$G_2=$\end{tabular}& \begin{tabular}{l}\tikz \node [scale=0.5, inner sep=0] {
   \begin{tikzpicture}
     \pic at (0,0) {G7b};
   \end{tikzpicture}
}; \end{tabular}& \begin{tabular}{l} $\pi_{G_2}^{(P_2)}(k)=\frac{7}{2}k^8-\frac{115}{2}k^7+\frac{807}{2}k^6-\frac{3123}{2}k^5+3582k^4-4843k^3+3547k^2-1074k$ \end{tabular}
\end{tabular}
\caption{Graphs with different degree sequences and matching chromatic pairs polynomials}\label{fig:G7s}
\end{figure}

We observe that the chromatic polynomial distinguishes fewer of our examples than the chromatic pairs polynomial does. 
The chromatic polynomial for a pseudotree $G$ on $n$ vertices with a unique cycle of length $\ell\le n$ is determined by $n$ and $\ell$ as $\pi_{G}(k)=\pi_{C_\ell}(k)\cdot (k-1)^{n-\ell}$, so the Figure~\ref{fig:G6s} graphs with different chromatic pairs polynomials share a chromatic polynomial $\pi_{R_1}(k)=\pi_{R_2}(k)=k^6-6k^5+14k^4-16k^3+9k^2-2k$. %
The graphs in Figure~\ref{fig:G7s} that are not distinguished by their chromatic pairs polynomial also share a chromatic polynomial, $\pi_{G_1}(k)=\pi_{G_2}(k)=k^7-12k^6+60k^5-159k^4+234k^3-180k^2+56k$. We conjecture that the chromatic pairs polynomials may be a more refined invariant than the chromatic polynomial. See Conjectures~\ref{conj:edge-implies-chromatic}~and~\ref{conj:deck-restraint-implies-chromatic} in Section~\ref{sect:closing-conj}.

\subsection{Coefficients of Chromatic Pairs Polynomials}
Motivated by the observations in Section \ref{sec:distinguish_trees}, we seek to establish basic structural properties of the chromatic pairs polynomial.  These properties can be used to rule out the possibility of two graphs having the same chromatic pairs polynomial.

Before stating and proving these properties, we note that deletion-contraction is useful for inductive proofs of structural properties of the chromatic polynomial and restrained chromatic polynomials (including proof of polynomiality).   
We have previously remarked that deletion-contraction does not hold in general for $\pi_G^{(H)}$ even though it does for $H=N_1$. For example, deletion-contraction for $G=P_2, H=P_2$ fails because the edge in $P_2$ is not present in either of the graphs for which $\rho_{G-v,r_v^{(2)}}$ is defined, and we observe that $\pi_{P_2}^{(P_2)}(k)=2\binom{k}{2}(k-2)\ne \pi_{P_2-e}^{(P_2)}(k)-\pi_{P_2/e}^{(P_2)}(k)=2\binom{k}{2}k-\binom{k}{2}$.  

However, we can still use deletion-contraction and other analysis of the individual $\rho_{r_v^{(2)}}$ to infer structural properties of the overall chromatic pairs polynomial, which we do in Theorems~\ref{thm:lowcoeffs}~and~\ref{thm:edge}. Both results consider coefficients of a chromatic pairs polynomial in the following standard form:
\begin{equation}\label{eq:edgestandard}
    \pi_G^{(P_2)}(k)=a_{n+1}k^{n+1}-a_n k^n + \dots + (-1)^{n} a_1k .
\end{equation}

We start with two easy observations. First, observe that the chromatic pairs polynomial for any graph $G$ is indeed a multiple of $k$ because of the $\binom{k}{2}$ term multiplied by restrained chromatic polynomials. Note that although any graph's chromatic polynomial is a multiple of $k$ \cite{Rea68}*{Theorem~9}, this is not true for restrained chromatic polynomials in general. For example, $N_1$ with vertex constraint $\set{1}$ has chromatic polynomial $k-1$. 

Second, applying deletion-contraction to the individual $\rho_{G-v,r_v^{(2)}}$ shows that, like chromatic polynomials, each $\rho_{G-v,r_v^{(2)}}$ is a polynomial of degree $\abs{V(G-v)}=n-1$ with coefficients with alternating signs \cite{Ere15}*{Theorem~4.1.2}. Because these are summed together and multiplied by $\binom{k}{2}$ in the chromatic pairs polynomial, the signs in the overall polynomial also alternate, and we have $a_i\ge 0$ for $i=1,\dots,n+1$ in Eq. \eqref{eq:edgestandard}.

Next we show that these polynomials are higher degree multiples of $k$ if their underlying graphs are disconnected. Theorem~\ref{thm:lowcoeffs} asserts that the lowest power of $k$ that appears with a nonzero coefficient is the number of connected components.

\begin{theorem}\label{thm:lowcoeffs}
Let $G$ be a graph with $t\ge 1$ connected components. Then the coefficients for $\pi_G^{(P_2)}$ in standard form (Eq.~\ref{eq:edgestandard}) are such that $a_i>0$ for $i\ge t$ and $a_{i}=0$ for $i<t$.
\end{theorem}

The proof of Theorem~\ref{thm:lowcoeffs} is given in Appendix~\ref{app:coefficients-positivity}. It first reasons that a constant term must be present in at least one restrained chromatic polynomial that appears in the definition of the chromatic pairs polynomial.  Subsequently by induction on the number of connected components, the proof uses the following lemma to show how the order of the lowest order nonzero coefficient increases with the number of connected components. We present the lemma and its simple proof below:

\begin{lemma}
\label{lemma:disjoint}
    Let $G_1$ and $G_2$ be two graphs, and let $G_1+G_2$ denote their disjoint union. Then,
    \begin{equation*}
        \pi_{G_1+G_2}^{(P_2)}(k)=\pi_{G_1}^{(P_2)}(k)\cdot \pi_{G_2}(k)+\pi_{G_2}^{(P_2)}(k)\cdot \pi_{G_1}(k)\ .
    \end{equation*}
    
\end{lemma}
\begin{proof} Every edge in a coloring graph for $G_1+G_2$ must correspond to either a vertex of $G_1$ or a vertex of $G_2$ changing color. Each of the $\pi_{G_1}^{(P_2)}(k)$ edges in $\mathcal C_k(G_1)$ appears $\pi_{G_2}(k)$ times in $\mathcal C_k(G_1+G_2)$, once per (independent) coloring of the vertices of $G_2$, and each of the $\pi_{G_2}^{(P_2)}(k)$ edges in $\mathcal C_k(G_2)$ appears $\pi_{G_1}(k)$ times in $\mathcal C_k(G_1+G_2)$.
\end{proof}

Having established how the number of connected components of a graph determines its chromatic pairs polynomial's low-degree coefficients, we now derive explicit expressions for the coefficients of the first few high-degree terms. 
It is well-known that for a graph $G$ with $\abs{V(G)}=n$ and $\abs{E(G)}=m$, the chromatic polynomial of  $G$ is monic of degree $n$ with $-m$ as the coefficient of $k^{n-1}$ \cite{Rea68}*{Theorems~7,8,11}. Whitney \cite{Whi32} showed that the remaining coefficients have a combinatorial interpretation in terms of broken circuits. The modern proof of Whitney's Theorem is in Sagan's textbook \cite{Sag20}*{Theorem~3.8.5}.

 Analogously, by \cite{Ere19}*{Theorems~4.2.1~and~4.2.2}, every restrained chromatic polynomial is a monic polynomial of degree $n$, whose coefficients alternate in sign.  Moreover, the magnitude of the second coefficient of a restrained chromatic polynomial is the number of edges in the underlying graph plus the total number of restraints. For each restrained polynomial $\rho_v^{(2)}$ in Eq.~\eqref{eq:edge_formula}, the underlying graph $G-v$ has $n-1$ vertices, $m-\deg(v)$ edges, and $2\deg(v)$ restraints. Hence each $\rho_v^{(2)}$ can be written in the form
    \begin{equation}
        \rho_v^{(2)}(k) = k^{n-1} -  (m-\deg(v)+2\deg(v))k^{n-2}+b_{n-3}^{(v)}k^{n-3}-\cdots+(-1)^{n-1} b_0^{(v)}\ ,
    \end{equation}
    with coefficients $b_j^{(v)}\ge 0$ for $j=0,\dots,n-3$. 
    Substituting this general form of $\rho_v^{(2)}$  into Eq.~\eqref{eq:edge_formula}, we obtain
    \begin{align}
    \label{eq:form_first2}
        \pi_G^{(P_2)}(k) & =   \binom{k}{2}\sum_{v\in V(G)}\left( k^{n-1}-(m+\deg(v))k^{n-2}+b_{n-3}^{(v)}k^{n-3}-\cdots+(-1)^{n+1} b_0^{(v)} \right) \nonumber\\ &= \frac{n}{2} k^{n+1}-\frac{n+nm+2m}{2} k^{n}+\frac{1}{2}\left( (n+2)m+\sum_{v\in V(G)}b_{n-3}^{(v)}\right)k^{n-1} %\sk{-a_{n-2}k^{n-1}+\dots+(-1)^na_1k\ ,}
        \\& \qquad - \frac{1}{2}\left(\sum_{v\in V(G)}\left(b_{n-3}^{(v)}+b_{n-4}^{(v)}\right)\right)k^{n-2}+\cdots+\frac{(-1)^{n}}{2}\left(\sum_{v\in V(G)}b_0^{(v)}\right)k \nonumber \ .
    \end{align}
This general formula for our chromatic pairs polynomial immediately gives expressions for the coefficients $a_{n+1}$ and $a_n$ in the chromatic pairs polynomial, which proves the first two results in the following theorem. 
 Additional results from \cite{Ere19} about restrained chromatic polynomials allow us to derive the next coefficient in the chromatic pairs polynomial as well. The proof of the $a_{n-1}$ coefficient in Theorem~\ref{thm:edge} below can be found in Appendix~\ref{app:third-coefficient}.
\begin{theorem}
\label{thm:edge}
    Let $G$ be a graph with $\abs{V(G)}=n$, $\abs{E(G)}=m$, degree sequence $\set{d_i}_{i\in[n]}$, and $\ell$ triangles. Then the coefficients for $\pi_{G}^{(P_2)}$ in standard form (Eq.~\ref{eq:edgestandard}) are such that
    \begin{align*}
        a_{n+1} &=  \frac{n}{2}, \quad a_{n} =  \frac{n+nm+2m}{2}, \quad \text{and} \quad a_{n-1} = \frac 1 2 \left( \frac{nm(m+1)}{2}+2m^2-m-(n+3)\ell+\frac 1 2 \sum_{i\in[n]}d_i^2 \right)\ .
    \end{align*}
\end{theorem}

This theorem shows us that the number of vertices and edges in a graph determine the first two coefficients of the chromatic pairs polynomial, and these quantities along with the number of triangles and the sum of the squares of the vertex degrees determine the third coefficient of the chromatic pairs polynomial. We can compare these results to the first three coefficients in the chromatic polynomial, which are 1, $-m$, and $\binom{m}{2}-\ell$, respectively \cite{Whi32}. This implies that graphs that share a chromatic polynomial must share the first two coefficients of the chromatic pairs polynomial. Additionally, graphs that share a chromatic pairs polynomial must agree on at least the first three coefficients of their chromatic polynomial.  This reinforces the idea that the chromatic pairs polynomial contains at least as much information as the chromatic polynomial, motivating Conjecture~\ref{conj:edge-implies-chromatic}, as $a_{n-1}$ includes extra information involving the degree sequence.

Theorem~\ref{thm:edge} connects to the graphs presented in the previous subsection as follows. 
Trees are triangle-free, and so the dependence of the first three coefficients on $n$, $m$, and degree sequence is consistent with Theorem~\ref{thm:treedeg}.  
Moreover, one can readily check that these graph properties are the same for the two pseudotrees in Figure \ref{fig:G6s}, which explains why the two chromatic pairs polynomials have the same $k^5$ coefficient even though subsequent coefficients diverge.  Interestingly, this also explains why the $k^6$ coefficients agree for the graphs in Figure~\ref{fig:G7s}. Although these graphs' degree sequences $\set{d_i}_{i\in[n]}$ are different, the sum of degrees squared $\sum d_i^2$ are the same, as well as number of vertices, edges, and triangles.

Finally, we note that knowledge only of the leading coefficients of chromatic pairs polynomials and chromatic polynomials ($n/2$ and 1, respectively) implies that chromatic pairs polynomials in general cannot be chromatic polynomials of different graphs. The only exceptions to this are graphs on two vertices, which have monic chromatic pairs polynomials that happen to coincide with chromatic polynomials of different graphs as follows:
\begin{align*}
\pi_{P_2}^{(P_2)}(k)&=\pi_{K_3}(k)=k^3-3k^2+2k \\ 
\pi_{N_2}^{(P_2)}(k)&=\pi_{N_1+P_2}(k)=k^3-k^2.
\end{align*}
The fact that leading coefficients of chromatic pairs polynomials differ in general from those of chromatic polynomials implies that the chromatic pairs polynomial is a new polynomial in the sense that there is no general mapping from graph $G$ to $G'$ such that the chromatic pairs polynomial for $G$ is the chromatic polynomial for $G'$.  

\subsection{Graphs with Unique Chromatic Pairs Polynomials}
The observation that the chromatic pairs polynomial for graph $G$ recovers the number of vertices and edges immediately (from the first two coefficients) shows that chromatic pairs polynomials are unique for null and complete graphs. 

Next, we discuss how chromatic pairs polynomial can help us detect whether a given graph is a tree. It is well-known  that a graph $G$ is a tree on $n$ vertices if and only if the chromatic polynomial $\pi_G(k)$ is equal to $k(k-1)^{n-1}$ \cite{Rea68}*{Theorem~13}. The next result is an analogue for the chromatic pairs polynomial. 

\begin{theorem} A graph $G$ on $n$ vertices with degree sequence $d_1, d_2, ..., d_n$ is a tree if and only if 
\begin{equation}\label{eq:tree}
\pi_{G}^{(P_2)}(k) = \binom{k}{2} \sum_{v\in G} (k-2)^{d_i} (k-1)^{n-d_i-1} \ .
\end{equation}
\end{theorem}

\begin{proof} If $G$ is a tree, then the chromatic pairs polynomial takes the desired form by Theorem~\ref{thm:treedeg}. Conversely, suppose graph $G$ satisfies Eq.~\ref{eq:tree}. Then its second coefficient reveals $\abs{V(G)}-1$ edges, and it has a nonzero linear term, implying $G$ is connected by Theorem~\ref{thm:lowcoeffs}. Such a graph must be a tree. 
\end{proof}

Theorem~\ref{thm:treedeg} also shows that chromatic pairs polynomials for paths and stars are unique, because they are the only trees with $2$ and $n-1$ leaves, respectively. This is despite paths and stars on $n$ vertices sharing a chromatic polynomial (with all other trees on $n$ vertices). 

Like null and complete graphs, cycles are known to be chromatically unique, meaning that for $n\ge 3$, $C_n$ is the only graph with its chromatic polynomial \cite{CW78}. Note that the uniqueness of the chromatic pairs polynomial for cycle graphs would follow immediately from Conjecture~\ref{conj:edge-implies-chromatic}. However, we will give an unconditional proof of this result.  

\begin{theorem}\label{thm:cycleuniqueness}
A graph $G$ is an $n$-cycle if and only if $\pi_{G}^{(P_2)}(k)=\pi_{C_n}^{(P_2)}(k)$.
\end{theorem}
\begin{proof}[Proof summary.]
Observe that any candidate graph with a matching chromatic pairs polynomial must be connected by Theorem~\ref{thm:lowcoeffs}, and because its chromatic pairs polynomial must match on the first two coefficients, it must have $n$ vertices and $n$ edges. In other words, it suffices to show that $C_n$ is the only pseudotree on $n$ vertices with its chromatic pairs polynomial. As an aside, this reasoning also allows us to quickly confirm chromatic uniqueness of cycle graphs, because any pseudotree on $n$ vertices with a cycle of length $\ell$ has chromatic polynomial $\pi_{G}(k)=\pi_{C_\ell}(k)(k-1)^{n-\ell}$, which only equals $\pi_{C_n}(k)$ when $\ell=n$. In fact, a graph $G$ on $n\geq 3$ vertices is a pseudotree with a cycle length $\ell$ with $3\leq \ell \leq n$ if and only if $\pi_G(k)=\pi_{C_{\ell}}(k) (k-1)^{n-\ell}$ \cite{Laz95}*{Theorem~2}.

We consider the formula for $\pi_G^{(P_2)}(k)$ for pseudotrees in Eq.~\eqref{eq:ptree_formula} and evaluate its derivative at $k=2$. Appendix~\ref{app:deriv} expands on this calculation, which results in: 
\[
\frac{d}{dk}\pi_G^{(P_2)}(k) \Big|_{k=2} = \begin{cases} \ell_2+n_1 & \ell\text{ even} \\ -\ell_2 & \ell\text{ odd}\end{cases},
\]
where $\ell_0$ is the number of cycle vertices of degree 2 and $n_1$ is the number of degree one vertices. Note that  this formula distinguishes the pseudotrees in Figure~\ref{fig:G6s}, the first of which has two degree-2 cycle vertices and the latter of which has only one. By contrast, a simple cycle has $$\frac{d}{dk}\pi_{C_n}^{(P_2)}(k)\mid_{k=2}=\begin{cases} n & n\text{ even} \\ -n & n\text{ odd}\end{cases}\ .$$ 
If a pseudotree is not a cycle graph, then $\ell_2<\ell_2+n_1<n$, establishing uniqueness of $\pi_{C_n}^{(P_2)}$. 
\end{proof}

\subsection{Hypercube Polynomials as Invariants}\label{sec:hyper}
Despite the uniqueness of certain chromatic pairs polynomials, other chromatic pairs polynomials are not able to distinguish non-isomorphic graphs, such as $T_2$ and $T_3$ in Figure~\ref{fig:trees} and $G_1$ and $G_2$ in Figure~\ref{fig:G7s}. We next investigate whether a different choice of $H$ could distinguish more pairs of non-isomorphic graphs. Like other reconfiguration systems, a graph recoloring problem can be described by a \emph{cube complex}~\cite{GP07}, the 1-skeleton of which is the coloring graph $\mathcal{C}_k(G)$. A cube complex is a (non-disjoint) union of  hypercubes, so generalizing the chromatic pairs polynomial to count higher-dimensional hypercubes may provide additional information about the coloring graph.

\begin{figure}[ht]
\begin{tikzpicture}[baseline=0cm, scale=0.5]
\draw[style=thick] (5,0)--(3,-2)--(5,-4)--(7,-2)--(5,0);

\draw[fill=white] (3.5,-.5) rectangle (6.5,.5);
\node[anchor=south] at (5,.5) {${c_{00}}$};
\draw[style=thick](4,0)--(6,0);
\draw[radius=.35,fill=white](4,0)circle node{1};
\draw[radius=.35,fill=black!20](5,0)circle;% node{3};
\draw[radius=.35,fill=white](6,0)circle node{1};
\node[anchor=east] at (3.75,0) {$v_0\rightarrow$};
\node[anchor=west] at (6.25,0) {$\leftarrow v_1$};

\draw[fill=white] (1.5,-2.5) rectangle (4.5,-1.5);
\node[anchor=east] at (1.5,-2) {${c_{01}}$};
\draw[style=thick](2, -2)--(4,-2);
\draw[radius=.35,fill=white](2,-2)circle node{2};
\draw[radius=.35,fill=black!20](3,-2)circle;
\draw[radius=.35,fill=white](4,-2)circle node{1};

\draw[fill=white] (5.5,-2.5) rectangle (8.5,-1.5);
\node[anchor=west] at (8.5,-2) {$c_{10}$};
\draw[style=thick](8, -2)--(6,-2);
\draw[radius=.35,fill=white](6,-2)circle node{1};
\draw[radius=.35,fill=black!20](7,-2)circle;
\draw[radius=.35,fill=white](8,-2)circle node{2}; 

\draw[fill=white] (3.5,-4.5) rectangle (6.5,-3.5);
\node[anchor=north] at (5,-4.5) {${c_{11}}$};
\draw[style=thick](4,-4)--(6,-4);
\draw[radius=.35,fill=white](4,-4)circle node{2};
\draw[radius=.35,fill=black!20](5,-4)circle;
\draw[radius=.35,fill=white](6,-4)circle node{2}; 
\end{tikzpicture} 

\caption{A 2-cube induced in $\mathcal C_3(P_3)$ by generator $(\set{v_0,v_1},\set{c_{00},c_{01},c_{10},c_{11}})$ with colors $k_{0,0}=1, k_{0,1}=2, k_{1,0}=1,k_{1,1}=2$, using the notation of Lemma~\ref{lem:cubeconstruction}}\label{fig:cube-generator}
\end{figure}
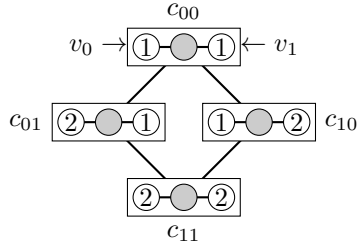

We have already described how to count occurrences of $C_4=Q_2$ in our section on counting small cycles. Now we focus more narrowly on the presence or absence of higher-dimensional hypercubes in coloring graphs. Lemma~\ref{lem:cubeconstruction} establishes necessary conditions for minimal $Q_s$-generators for any $s\ge 0$. At a high level, the lemma states that an induced $s$-cube in a coloring graph can only arise from exactly $s$ vertices independently swapping between two colors each. Figure~\ref{fig:cube-generator} illustrates this phenomenon  by annotating an induced $2$-cube in $\mathcal{C}_3(P_3)$ that arises from the choice of the left and right vertices of the base graph for $U=\set{v_0,v_1}$, respectively. The four partial colorings in $C=\set{c_{00},c_{01},c_{10},c_{11}}$ are enumerated in binary to make it easier to read off the relevant bits $\lfloor{i/2^j}\rfloor\pmod{2}$ determining the color of vertex $j$. For example,  $c_{01}$ assigns color $k_{0,1}=2$ to vertex $v_0$ because $\lfloor (01)_2/2^0\rfloor \pmod{2}= 1$, and it assigns color $k_{1,0}=1$ to vertex $v_1$ because $\lfloor (01)_2/2^1\rfloor \pmod{2} = 0$.

\begin{lemma}\label{lem:cubeconstruction}
    For graph $G$ and $s\ge 0$, every minimal $Q_s$-generator $(U,C)$ is of the form $U=\set{v_0,\dots,v_{s-1}}$ and $C=\set{c_{0},\dots,c_{2^s-1}}$ for some positive integers $k_{0,0}<k_{0,1},\dots,k_{s-1,0}<k_{s-1,1}$ such that $c_i(v_j)=k_{j,\lfloor\frac{i}{2^j}\rfloor\pmod{2}}$ for all $i=0,\dots,2^s-1$, $j=0,\dots,s-1$.
\end{lemma}

\begin{proof}
    This is trivially true for $H=N_1=Q_0$, where the chromatic $H$-polynomial is the chromatic polynomial, generated by $(\emptyset,c_\emptyset)$, where $c_{\emptyset}$ is the unique coloring with empty domain. For any $s>0$, suppose the theorem is true for induced $H=Q_{s-1}$. Consider $H=Q_s$ and note that $H$ can be partitioned into two disjoint induced copies of $Q_{s-1}$. By induction and our observation about monotonicity in $H$, any minimal $Q_s$-generator contains a minimal $Q_{s-1}$-generator $(\set{v_0,\dots,v_{s-2}},\set{c_0^{(0)},\dots,c_{2^{s-1}-1}^{(0)}})$. Fix a partition of the vertices of $Q_s$ such that each part induces $Q_{s-1}$. By the structure of $Q_s$, each vertex in the first part is associated with some $c_{i}^{(0)}$ that has an edge to a vertex in the other part that we will call $c_{i}^{(1)}$. This edge cannot be due to any vertex in $\set{v_0,\dots,v_{s-2}}$ taking on a third color, because that would create a triangle that is not present in $Q_s$, so let $v_{s-1}$ be the new color-changing vertex. For any $i'=0,\dots,2^{s-1}-1$ that differs from $i$ by a power of 2 (so it is neighboring in its $Q_{s-1}$ subgraph) the edge from $c_{i'}^{(0)}$ to $c_{i'}^{(1)}$ must correspond to the same $v_{s-1}$ changing color so that $c_{i}^{(1)}$ and $c_{i'}^{(1)}$ can be neighbors. {The order of the two colors for $v_{s-1}$ is irrelevant, so label them $k_{{s-1},0}<k_{s-1,1}$.} This forces the $Q_s$-generator to be of the desired form, completing our proof.
 \end{proof}

This lemma allows us to establish necessary and sufficient conditions for presence of hypercubes in a coloring graph in Theorem~\ref{thm:hypercubes}. Figure~\ref{fig:4cubes} illustrates conditions that establish presence of an induced 5-cube in the 5-coloring graph of one of the graphs from Figure~\ref{fig:G7s}. In particular, the left annotation indicates the $s=5$ non-gray vertices that constitute $U$, and their labeling indicates a 2-coloring $c$. Given $U$ and $c$, define the following restraint on $G-U$: 
\begin{equation}\label{eq:hypercube-r}
r_{(U,c)}(v) = \set{j:2c(u)-1=j\text{ or }2c(u)=j\text{ for some }u\in N(v)\cap U}\text{ for any $v\in V(G)-U$ .}
\end{equation}
Then the labels on the right confirm that $\rho_{G-U,r_{(U,c)}}(k)>0$ for $k=5$, so indeed $\pi_G^{(Q_5)}(5)>0$.

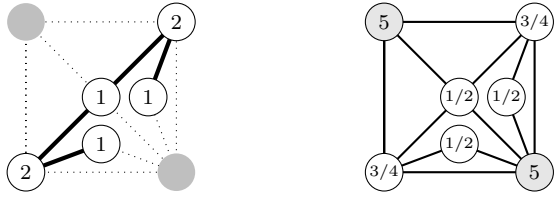
\begin{figure}[ht]
%\tikz\node[scale=.5, inner sep=0]{
\begin{tikzpicture}[scale=.5, baseline=0cm]
\node (a) at (0,0) {};
\node (b) at (4,0) {};
\node (c) at (4,4) {};
\node (d) at (0,4) {};
\node (e) at (2,.75) {};
\node (f) at (3.25,2) {};
\node (g) at (2,2) {};
\draw[style=dotted] (b)--(d)--(a)--(b)--(c)--(d)--(a);
\draw[style=dotted] (e)--(b)--(f);
\draw[style=ultra thick] (e)--(a)--(c)--(f);
\footnotesize{\draw[radius=.5,fill=white](a) circle node{2};
\draw[radius=.5,draw=none,fill=black!25](b) circle node{};
\draw[radius=.5,fill=white](c) circle node{2};
\draw[radius=.5,draw=none,fill=black!25](d) circle node{};
\draw[radius=.5,fill=white](e) circle node{1};
\draw[radius=.5,fill=white](f) circle node{1};
\draw[radius=.5,fill=white](g) circle node{1};}
\end{tikzpicture}
%};
\qquad\qquad\qquad
%\tikz\node[scale=.5, inner sep=0]{
\begin{tikzpicture}[scale=.5, baseline=0cm]
\node (a) at (0,0) {};
\node (b) at (4,0) {};
\node (c) at (4,4) {};
\node (d) at (0,4) {};
\node (e) at (2,.75) {};
\node (f) at (3.25,2) {};
\node (g) at (2,2) {};
\draw[style=thick] (b)--(d)--(a)--(b)--(c)--(d)--(a);
\draw[style=thick] (b)--(e)--(a)--(c)--(f)--(b);
\footnotesize{
\draw[radius=.5,fill=black!10](b) circle node{5};
\draw[radius=.5,fill=black!10](d) circle node{5};}
\tiny{\draw[radius=.5,fill=white](a) circle node{3/4};
\draw[radius=.5,fill=white](c) circle node{3/4};
\draw[radius=.5,fill=white](e) circle node{1/2};
\draw[radius=.5,fill=white](f) circle node{1/2};
\draw[radius=.5,fill=white](g) circle node{1/2};}
\end{tikzpicture}
%};
    \caption{A 2-coloring of a 5-vertex induced subgraph of graph $G$ and the corresponding $2^5$ colorings of $G$ that induce a copy of $Q_5$ in $\mathcal C_5(G)$}\label{fig:4cubes}
\end{figure}

\begin{theorem}\label{thm:hypercubes}
    
    For graph $G$, $\pi_G^{(Q_s)}(k)>0$ if and only if there exists some $U$ and $c$ such that $U\subseteq V$ with $\abs{U}=s$, $c:U\to[\lfloor k/2\rfloor]$ is a proper coloring of $G[U]$, and $\rho_{G-U,r_{(U,c)}}(k)>0$ for $r_{(U,c)}$ as in Eq.~\eqref{eq:hypercube-r}. 
\end{theorem}

\begin{proof}
For the reverse direction, consider a budget of $k$ colors, and suppose $c$ uses at most half of them to color $U\subseteq G(V)$. Then the two colorings $2c-1$ and $2c$ each properly color  $U$, and so does any $c'$ with $c'(v)=2c(v)-1$ or $c'(v)=2c(v)$ for each $v$.  If $k$ colors suffice to color the rest of the graph respecting the restraint $r_{(U,c)}$, then we can fix any one such coloring of the rest of the graph and induce $Q_{\abs{U}}$ by independently coloring the vertices of $U$ with any combination of the colors allowed by $2c-1$ and $2c$. 

Conversely, a positive count for $s$-cubes in a $k$-coloring graph implies some minimal $Q_s$-generator $(U,C)$ such that $\rho_{G-U,r_{(U,C)}}(k)>0$. 
By Lemma~\ref{lem:cubeconstruction}, $U=\set{v_0,\dots,v_{s-1}}$ and $C$ consists of every coloring that assigns each $v_i\in U$ to one of two colors $k_{i,0}$ or $k_{i,1}$. Among all such colorings, choose the $c^*\in C$ that uses the smallest color palette $\mathcal P=\set{j_1,\dots,j_\kappa}$ with $j_1<\dots <j_\kappa$. Note that $\kappa\le \lfloor k/2\rfloor$ because $C$ uses at least $2\kappa$ colors and $\rho_{G-U,r_{(U,C)}}(k)>0$. Then for $\sigma_\mathcal P$ as in Eq.~\eqref{eq:palette-perm}, let $c$ be the coloring on $U$ defined by $c(v_i)=\sigma_{\mathcal P}(c^*(v_i))$ for all $v_i\in U$.  Then for $C'$ the set of all colorings $c'$ that assign each $v\in U$ either $2c(v)-1$ or $2c(v)$, $(U,C')$ must also be a $Q_s$-generator with $\rho_{G-U,r_{(U,C')}}(k)=\rho_{G-U,r_{(U,c)}}(k)>0$ because the colors for the hypercube respect the adjacency within $U$ and the hypercube uses no more colorings than that induced by our original $(U,C)$.
\end{proof}

If instead we are satisfied with a simple sufficient condition for presence of a particular type of hypercube, we can restrict our focus to independent sets in a base graph, which we do in the following corollary.
 In particular, if the set $U$ being recolored in $Q_s$ is independent, then it is properly colored by the constant coloring $c(v)=1$, so if the rest of the graph is $\ell$-colorable, then $\ell+2$ colors guarantee a coloring for the whole graph with color swaps for $U$ yielding $Q_s$. 
 
\begin{corollary}\label{cor:hypercube}
        If graph $G$ has an independent set $U$ with $\abs{U}=s$ and $\chi(G-S)=\ell$, then $\pi_G^{(Q_s)}(k)>0$ for all $k\ge \ell+2$. 
\end{corollary}

We now consider how this theorem and corollary do and do not serve as invariants for pairs of graphs studied previously. Our corollary shows that although chromatic pairs polynomials (counting $Q_1=P_2$) cannot distinguish the trees $T_2$ and $T_3$ in Figure~\ref{fig:trees}, higher-dimensional hypercube polynomials can. 
To see why, note that $T_3$ contains an independent set of size four, % whereas $T_3$ does not. 
so $\mathcal C_3(T_3)$ contains 4-dimensional hypercubes corresponding to this independent set swapping between two colors with the rest of the graph fixed as the third color. For any four vertices of $T_2$ to participate in a 4-dimensional hypercube in $\mathcal C_3(T_2)$, they would have to be independent (not possible in $T_2$), or neighboring vertices would have to have disjoint pairs of colors (not possible with only $k=3$ colors).

 However, even the more general Theorem~\ref{thm:hypercubes} cannot help distinguish our examples on seven vertices, because cubes of each dimension first appear at the same $k$ for each graph. Both this theorem and its corollary are rather coarse invariants in that they establish presence or absence of certain features in the coloring graph, not the counts that we have focused on up until this point. We can verify computationally that although squares appear at $k=4$ in the coloring graphs for each of those graphs, the actual square count as determined by Eq.~\eqref{eq:C4} differs. Appendix~\ref{app:hypercubes} fully justifies these claims.

In our final section below, we conjecture more broadly that any pair of non-isomorphic graphs admit some $H$ for which the corresponding $H$-polynomials are different. 
Although hypercubes are natural candidates to study, our conjecture does not require $H$ to be a hypercube.

\section{Closing Conjectures}\label{sect:closing-conj}
We have shown that particular instantiations of our (chromatic) $H$-polynomial serve as refinements of the chromatic polynomial as a graph invariant, at least for particular classes of graphs. In particular, the chromatic pairs polynomial refines the chromatic polynomial for trees. We conjecture that the same holds for all graphs.

\begin{conjecture}\label{conj:edge-implies-chromatic}
    If two graphs $G_1$ and $G_2$ satisfy $\pi_{G_1}^{(P_2)}(k) = \pi_{G_2}^{(P_2)}(k)$, then $\pi_{G_1}(k) = \pi_{G_2}(k)$.
\end{conjecture}

Since $\sum_{v\in G} \rho_{r_v^{(2)}}(k)$ 
 determines $\pi_G^{(P_2)}$  (Equation~\eqref{eq:edge_formula}), we also pose the following weaker conjecture. 

\begin{conjecture}\label{conj:deck-restraint-implies-chromatic}
    If two graphs $G_1$ and $G_2$ share the same multiset $\{\rho_{r_v^{(2)}}(k) \}_{v\in V(G_1)}= \{\rho_{r_v^{(2)}}(k) \}_{v\in V(G_2)}$, then $\pi_{G_1}(k) = \pi_{G_2}(k)$.
\end{conjecture}

Conjecture~\ref{conj:deck-restraint-implies-chromatic} is similar to the \emph{Polynomial Reconstruction Problem} (PRP) \cite{Sch77} which states that $\pi_G(k)$ can be recovered from the multiset $\{\pi_{G-v}(k)\}_{v\in V(G)}$. 
Conjectures~\ref{conj:deck-restraint-implies-chromatic}
 and PRP differ in the information contained in each of the decks, and it is unclear which conjecture is more tractable. Our restrained chromatic polynomials may contain additional information compared to the (unrestrained) chromatic polynomials on the same subgraphs because the restraints give some indication about how the corresponding $v$ are connected to the rest of the graph; it is unclear whether this information is useful for the problem of reconstructing the chromatic polynomial of the overall graph.

We may further generalize Conjecture~\ref{conj:edge-implies-chromatic} to ask whether there may be a partial ordering on subgraphs $H$ defined by $H_1\le H_2$ for any $H_1,H_2$ such that every pair of graphs distinguished by their $H_1$-polynomials is also distinguished by their $H_2$-polynomials. Even if there is no universal $H$ that is a complete invariant for all graphs, it is conceivable that for any given graph $G$, there may be specific choices of $H$ for which the polynomials $\pi_{G}^{(H)}(k)$ determine $G$. In all our examples of non-isomorphic graphs, we observed that every pair differed on some $H$-polynomial, which implies that they differed on the family of coloring graphs. We propose that the family of coloring graphs can function as a complete graph invariant. This prediction is equivalent to a conjecture stated fully in terms of $H$-polynomials. If two coloring graphs are isomorphic, then they will have the same counts for all $H$. Conversely, if $\mathcal{C}_k(G_1)$ and $\mathcal{C}_k(G_2)$ have the same number of induced copies of $H$ for all $H$, we first take $H=N_1$ to see that $\mathcal{C}_k(G_1)$ and $\mathcal{C}_k(G_2)$ share the same number of vertices.   Then for each $k$, taking $H=\mathcal{C}_k(G_1)$, we see that $\pi_{G_1}^{(\mathcal C_k(G_1))}(k)=\pi_{G_2}^{(\mathcal C_k(G_1))}(k)=1$, which implies that $\mathcal{C}_k(G_1)\cong \mathcal{C}_k(G_2)$.
 We conjecture that either of these equivalent conditions is a complete graph invariant.
 
\begin{conjecture}\label{conj:col-graphs-determine-G}
For any graph $G$, the collection $\set{\mathcal C_k(G)}_{k\ge 1}$ uniquely determines $G$. Equivalently stated, graphs $G_1$ and $G_2$ are isomorphic if and only if $\pi_{G_1}^{(H)}=\pi_{G_2}^{(H)}$ for every $H$.

\end{conjecture}

We also state a stronger conjecture where only finitely many coloring graphs are needed. We suspect that Conjecture~\ref{conj:col-graphs-determine-G} may be equivalent to the seemingly stronger Conjecture~\ref{conj:finite-col-graphs-determine-G}. Given that the graph $G$ is finite, we expect that finitely many colors suffice to accurately capture the structure of $G$.

\begin{conjecture}\label{conj:finite-col-graphs-determine-G}
    There exists some function $f\colon \text{Graphs} \to\mathbb N$ that maps finite graphs to natural numbers such that for any graph $G$, the collection $\set{\mathcal C_k(G)}_{k=1}^{f(G)}$ uniquely determines $G$.
\end{conjecture}

Subsequent to a posted preprint of this work, \cite{HSTT24} proved Conjecture~\ref{conj:finite-col-graphs-determine-G} for $f(G)=5\abs{V(G)}^2+1$. Therefore, both Conjectures~\ref{conj:finite-col-graphs-determine-G}~and~\ref{conj:col-graphs-determine-G} have now been proven. Although coloring graphs with sufficiently large $k$ relative to $G$ can be used to recover $G$, the more general framing of partial orderings on $H$-polynomials still leaves much to be explored.

\medskip 

\textbf{Acknowledgments.}
We are grateful to the referees for their valuable comments and suggestions on this manuscript.

\section*{Statements and Declarations}
The authors declare that no funds, grants, or other support were received during the preparation of this manuscript.  The authors have no relevant financial or non-financial interests to disclose. 
All authors contributed equally to the research and preparation of the manuscript.

%%%%%%%%%%%%%%%%%%%%%%%%%%%%%%%%%%%%%%%%%%%%%%%%%%%%%%%%%%%%%%%

\appendix
\section{6-Cycle Polynomial}\label{app:6cycles}
Subsection~\ref{sec:6cycles} describes the 2- and 3-vertex generators needed to instantiate Theorem~\ref{thm:general} for $\pi_{G}^{(C_6)}$. We enumerate the associated restraints below in order to provide a closed-form expression in Eq.~\eqref{eq:C6}.

\begin{itemize}
    \item $s_{uv}^{(6)}$ restrains $u$'s neighbors with $\set{1,2,3}$ and $v$'s neighbors with $\set{4,5,6}$; 
    \item $s_{uv}^{(5)}$ restrains $u$'s neighbors with $\set{1,2,3}$ and $v$'s neighbors with $\set{1,4,5}$; 
    \item $s_{uv}^{(4)}$ restrains $u$'s neighbors with $\set{1,2,3}$ and $v$'s neighbors with $\set{1,2,4}$; 
\item $s_{uv}^{(3)}$ restrains $u$'s and $v$'s neighbors with $\set{1,2,3}$;
        \item $s_{uvw}^{(6)}$ restrains $u$'s neighbors with $\set{1,2}$, $v$'s neighbors with $\set{3,4}$, and $w$'s neighbors with $\set{5,6}$;
    \item $s_{uvw}^{(5)}$ restrains $u$'s neighbors with $\set{1,2}$, $v$'s neighbors with $\set{1,3}$, and $w$'s neighbors with $\set{4,5}$; 
    \item $s_{uvw}^{(4a)}$ restrains $u$'s and $v$'s neighbors with $\set{1,2}$ and $w$'s neighbors with $\set{3,4}$;
    \item $s_{uvw}^{(4b)}$ restrains $u$'s neighbors with $\set{1,2}$, $v$'s neighbors with $\set{1,3}$, and $w$'s neighbors with $\set{2,4}$; 
    \item $s_{uvw}^{(4c)}$ restrains $u$'s neighbors with $\set{1,2}$, $v$'s neighbors with $\set{1,3}$, and $w$'s neighbors with $\set{1,4}$;
    \item $s_{uvw}^{(3a)}$ restrains $u$'s neighbors with $\set{1,2}$, $v$'s neighbors with $\set{1,3}$, and $w$'s neighbors with $\set{2,3}$; 
    \item $s_{uvw}^{(3b)}$ restrains $u$'s and $v$'s neighbors with $\set{1,2}$ and $w$'s neighbors with $\set{1,3}$; and
    \item $s_{uvw}^{(2)}$ restrains $u$'s, $v$'s, and $w$'s neighbors with $\set{1,2}$.
\end{itemize}

Then by considering which colorings of $u,v,w$ are possible depending on the edges within those vertices and appropriately counting symmetries, we have the following closed-form expression: 
\begin{align}\label{eq:C6}
    \pi_{G}^{(C_6)}(k) &= \sum_{uv\in E(G)}\binom{k}{3}\rho_{s_{uv}^{(3)}}(k) + \sum_{\substack{\set{u,v}\in \binom{V(G)}{2} \\ uv\not\in E(G)}} 6\binom{k}{3} \rho_{s_{uv}^{(3)}}(k)\\
    & \quad + \sum_{uv\in E(G)}12\binom{k}{4} \cdot \rho_{s_{uv}^{(4)}}(k) + \sum_{\substack{\set{u,v}\in \binom{V(G)}{2} \\ uv\not\in E(G)}} 72\binom{k}{4} \rho_{s_{uv}^{(4)}}(k) \notag \\
    & \quad + \sum_{uv\in E(G)}60\binom{k}{5} \cdot \rho_{s_{uv}^{(5)}}(k) + \sum_{\substack{\set{u,v}\in \binom{V(G)}{2} \\ uv\not\in E(G)}} 180\binom{k}{5} \rho_{s_{uv}^{(5)}}(k) \notag \\
        & \quad + \sum_{\set{u,v}\in \binom{V(G)}{2}} 120\binom{k}{6} \rho_{s_{uv}^{(6)}}(k) \notag \\
& \quad + \sum_{\set{u,v,w}\in \binom{V(G)}{3}} 360\binom{k}{6} \rho_{s_{uvw}^{(6)}}(k) \notag \\ & \quad + \sum_{\substack{\set{u,v,w}\in \binom{V(G)}{3}\\ uv\not\in E(G)}} \left[ 240 \binom{k}{5} \rho_{s_{uvw}^{(5)}}(k) + 24\binom{k}{4} \rho_{s_{uvw}^{(4a)}}(k) \right] \notag \\
    & \quad +  \sum_{\substack{\set{u,v,w}\in \binom{V(G)}{3}\\ uv,uw\not\in E(G)}} 96\binom{k}{4} \rho_{s_{uvw}^{(4b)}}(k)\notag \\
    & \quad + \sum_{\substack{\set{u,v,w}\in \binom{V(G)}{3}\\ uv,uw,vw\not\in E(G)}} \left[ 96 \binom{k}{4} \rho_{s_{uvw}^{(4c)}}(k) + 26\binom{k}{3} \rho_{s_{uvw}^{(3a)}}(k) + 48\binom{k}{3}\rho_{s_{uvw}^{(3b)}}(k) + 4\binom{k}{2}\rho_{s_{uvw}^{(2)}}(k)\right]. \notag 
\end{align}

\section{Proof of Theorem \ref{thm:lowcoeffs}}
\label{app:coefficients-positivity}

\begin{proof}[Proof of Theorem~\ref{thm:lowcoeffs}]
First, we show that $\rho_{G, r}$ has alternating nonzero coefficients (including the constant term) for any connected graph $G$ and any restraint $r$ that imposes restraint $\set{1,2}$ on at least one vertex and imposes no other restraints. We proceed by induction on the number of edges in $G$. If $G$ is edgeless, then $G=N_1$ and $\rho_{G,r}(k)=k-2$. If $G$ has $m>0$ edges, then $\rho_{G,r}(k)=\rho_{G-e,r}(k)-\rho_{G/e,r_e}(k)$, where $r_e$ denotes the restraint on $G/e$ that unions the restraints of $r$ on the ends of $e$ for the contracted vertex in $G/e$. As $G/e$ is connected, it has a nonzero constant term by induction. Because $G-e$ has one more vertex than $G/e$, the signs of the constant terms in the respective expressions will differ; taking the difference ensures that the resulting polynomial still has alternating nonzero coefficients.

We now reason about the coefficient of the linear term in the chromatic pairs polynomial for a connected graph $G$. Select a vertex $w\in V(G)$ that is not a cut vertex so that $G-w$ is connected. By our earlier observation, $\rho_{G-w,r_w^{(2)}}(k)$ has a nonzero constant term, noting that $\rho_{G-w,r_w^{(2)}}(k)=1$ if $G=N_1$. Then our chromatic pairs polynomial $\pi_{G}^{(P_2)}(k)=\binom{k}{2}\sum_{v\in V(G)} \rho_{G-v,r_v^{(2)}}(k)$
has no constant term and a nonzero coefficient for $k$. Moreover, each of the $\rho_{G-v,r_v^{(2)}}$ have the same degree with alternating nonzero coefficients of the same sign. Since $\binom{k}{2} = \frac{k^2-k}{2}$ has alternating coefficients, it follows that $\pi_{G}^{(P_2)}(k)$ has alternating nonzero coefficients as well.

Next, suppose the result is true for all graphs with $t\ge 1$ connected components, and suppose $G$ on $n$ vertices has $t+1$ connected components, with $G=G_1+G_2$ where $G_1$ has $t$ connected components on $n_1$ vertices and $G_2$ is connected on $n_2$ vertices. By Lemma~\ref{lemma:disjoint}, we have
\[
\pi_{G_1+G_2}^{(P_2)}(k)=\pi_{G_1}^{(P_2)}(k)\cdot \pi_{G_2}(k)+\pi_{G_2}^{(P_2)}(k)\cdot \pi_{G_1}(k)\ ,
\]
and by induction, the lowest degree terms with nonzero coefficients in the four polynomials are for $k^{t},k,k,k^t$, respectively. Moreover, the signs of the products of each of these terms is $(-1)^{n-t}$, so they cannot cancel each other out.

Hence, the chromatic pairs polynomial of a graph with $t$ connected components has $a_t>0$ and $a_{t'}=0$ for each $t'<t$. Moreover, the coefficients $a_t$ for $t'>t$ are nonzero and alternating in sign. Indeed, the four polynomials $\pi_{G_1}^{(P_2)}(k)$, $\pi_{G_2}(k)$, $\pi_{G_2}^{(P_2)}(k)$ and $\pi_{G_1}(k)$ have nonzero alternating coefficients, and so do the products $\pi_{G_1}^{(P_2)}(k)\cdot \pi_{G_2}(k)$ and $\pi_{G_2}^{(P_2)}(k)\cdot \pi_{G_1}(k)$ which have matching degrees. This shows $\pi_{G_1+G_2}^{(P_2)}(k)$ has nonzero alternating coefficients, too.
\end{proof}

\section{Proof of Theorem \ref{thm:edge}}
\label{app:third-coefficient}

\begin{proof}[Proof of Theorem~\ref{thm:edge}]
The proof of the first two coefficients $a_{n+1}$ and $a_n$ was given before the theorem statement.  
In order to calculate the third highest coefficient $a_{n-1}$ of the chromatic pairs polynomial, we will use \cite{Ere19}*{Theorem~4.2.3}, which gives an expression for the third highest coefficient for a restrained chromatic polynomial.  For a graph $G$ with $n$ vertices, $m$ edges, and restraints $r$, the restrained chromatic polynomial is a polynomial of the form $\rho_{G,r}(k)=k^n-b_{n-1}k^{n+1}+b_{n-2}k^{n-2}-\cdots (-1)^{n+1}b_0$.  Then the coefficient $b_{n-2}$ can be written as
\begin{equation}
\label{eq:third_erey}
    b_{n-2} = \binom{m}{2}-\operatorname{Tri}(G)+\sum_{i<j}|r(v_i)||r(v_j)|+m\sum_{v\in V(G)}|r(v_i)|-\sum_{v_iv_j\in E(G)}|r(v_i)\cap r(v_j)| \ ,
\end{equation}
where $\operatorname{Tri}$ counts the number of triangles in $G$.

We now apply this formula to each instance of $\rho_{r_v^{(2)}}$ in Eq.~\eqref{eq:edge_formula}.  Instantiating  \eqref{eq:third_erey} with the graph $G-v$ with restraints \{1,2\} on all neighbors of $v$ yields
\begin{align*}
    b_{n-3}^{(v)}&=  \binom{m-\deg(v)}{2} -\operatorname{Tri}(G-v)+4\binom{\deg(v)}{2}+2\deg(v)(m-\deg(v))-2\operatorname{Tri}_v(G) \\ 
    &=  \binom{m}{2}+\deg(v)\frac{2m-3}{2}+\frac{1}{2}\deg(v)^2-\operatorname{Tri}(G-v)-2\operatorname{Tri}_v(G)\ ,
\end{align*}
where $\operatorname{Tri}_v(G)$ counts the number of triangles in $G$ that include vertex $v$.  We have also rewritten this coefficient as $b_{n-3}^{(v)}$ as the graph $G-v$ has $n-1$ vertices.  As in the chromatic pairs polynomial formula, we now sum $b_{n-3}^{(v)}$ over all vertices $v$.  To simplify the terms involving the triangle counts, we note that
$$
\sum_{i=1}^{n} \operatorname{Tri}_{v_i}(G) = 3 \operatorname{Tri}(G) \ ,
$$
since each triangle gets counted $3$ times on the left hand side. Moreover,
$$
\sum_{i=1}^{n} \operatorname{Tri}(G- v_i) = (n-3) \operatorname{Tri}(G) \ ,
$$
since each triangle get counted exactly $n-3$ times on the left hand side (one for each vertex $v$ not part of the triangle).  Letting $d_i=\deg(v_i)$, we have
\begin{align*}
    \sum_{i=1}^n b_{n-3}^{(v_i)} = &\sum_{i=1}^n \left[\binom{m}{2}+d_i\frac{2m-3}{2}+\frac{1}{2}d_i^2-\operatorname{Tri}(G-v)-2\operatorname{Tri}_v(G)\right] \\
    = & -(n+3)\operatorname{Tri}(G) + n\binom{m}{2} +m(2m-3) +\frac{1}{2}\sum_{i=1}^n d_i^2 \ ,
\end{align*}
where we have used the handshake lemma $2m=\sum_{i\in[n]}d_i$.
To finish our computation of the third highest coefficient in the chromatic pairs polynomial $\pi^{(P_2)}_{G}(k)$, we need to compute the coefficient of $k^{n-1}$ in the expansion $\binom{k}{2} \cdot \sum_{v\in V(G)} \rho_v(k)$.  From \eqref{eq:form_first2}, we need to add $nm+2m$ to our result and divide by 2.  Thus, the coefficient of 
$k^{n-1}$ in $\pi^{(P_2)}_{G}(k)$ is
$$
\frac{1}{2}\left(nm+2m - (n+3)\operatorname{Tri}(G) + \frac{nm(m-1)}{2} + 2m^2 - 3m + \frac{1}{2} \sum_{i=1}^{n} d_i^2 \right) \ ,
$$
which can be simplified to
$$
\frac{1}{2} \left( \frac{nm(m+1)}{2} + 2m^2 - m - (n+3) \operatorname{Tri}(G) + \frac{1}{2}\sum_{i=1}^{n} d_i^2 \right) \ .
$$
\end{proof}

\section{The Derivative of the Chromatic Pairs Polynomials for Pseudotrees}\label{app:deriv}
For pseudotree $G$, we want to calculate the derivative $\frac{d}{dk} \pi_{G}^{(P_2)}(k)$ evaluated at $k=2$. As a preliminary, we calculate $\pi_{C_{\ell}}(k)$ and $\frac{d}{dk} \pi_{C_{\ell}}(k)$ at $k=2$. Since $\pi_{C_{\ell}}(k) = (k-1)^{\ell} + (-1)^{\ell} (k-1)$, we get $\pi_{C_{\ell}}(2)=1+(-1)^{\ell}$ which reflects the fact that $C_{\ell}$ is $2$-colorable if and only if $\ell$ is even. Next,
\begin{align*}
\frac{d}{dk} \pi_{C_{\ell}}(k) = \ell (k-1)^{\ell-1} + (-1)^{\ell}  \ \Rightarrow \ \frac{d}{dk} \pi_{C_{\ell}}(k)\Big|_{k=2} = \ell + (-1)^{\ell} 
\end{align*}
For the derivative of $\pi^{(P_2)}_{C_{\ell}}$, we have
\begin{align*}
&\frac{d}{dk} \pi^{(P_2)}_{C_{\ell}}(k) = 
\frac{\ell}{2} (k-1)^{\ell-2} \left[2(k-1)(k-2) + (\ell-1) k(k-4) \right] 
+ 2 \ell(\ell-1) (k-1)^{\ell-2} + (-1)^{\ell} \ell (2k-3)  \\ 
\Rightarrow & \ \ \ \frac{d}{dk} \pi^{(P_2)}_{C_{\ell}}(k)\Big|_{k=2} = (-1)^{\ell} \cdot \ell 
\end{align*}
Moreover, $\pi^{(P_2)}_{C_{\ell}}(2)=0$. Next, we calculate $\frac{d}{dk} \pi_{G}^{(P_2)}(k)$ for a pseudotree $G$. We have
$$
\frac{d}{dk} \pi_{G}^{(P_2)}(k) = \eqref{eq:d1} + \eqref{eq:d2} + \eqref{eq:d3} + \eqref{eq:d4} + \eqref{eq:d5}
\ ,$$
where the individual terms are as follows: 
\begin{align}
 &\frac{1}{\ell}\cdot \frac{d}{dk} \pi^{(P_2)}_{C_{\ell}}(k) \cdot \sum_{i=1}^{\ell} (k-2)^{d_i-2} (k-1)^{n-\ell-(d_i-2)}\ , 
   \label{eq:d1}  \\ 
 &\frac{\pi^{(P_2)}_{C_{\ell}}(k)}{\ell} \sum_{i=1}^{\ell} (d_i-2) (k-2)^{d_i-3}(k-1)^{n-\ell-(d_i-2)} + (n-\ell-(d_i-2)) (k-2)^{d_i-2} (k-1)^{n-\ell-d_i+1}\ ,%{n-\ell-(d_i-2)-1}, 
 \label{eq:d2}\\ 
&\binom{k}{2} \cdot \frac{1}{k}\cdot \frac{d}{dk} \pi_{C_{\ell}}(k) \cdot
\sum_{k=\ell+1}^{n} (k-2)^{d_i} (k-1)^{n-\ell-d_i}\ , \label{eq:d3}\\
& \binom{k}{2} \cdot \frac{\pi_{C_{\ell}}(k)}{k}\cdot 
\sum_{k=\ell+1}^{n} d_i (k-2)^{d_i-1} (k-1)^{n-\ell-d_i} + (n-\ell-d_i) (k-2)^{d_i} (k-1)^{n-\ell-d_i-1}\  \label{eq:d4}\\
& \frac{1}{2} \pi_{C_\ell}(k)\sum_{i=\ell+1}^n (k-2)^{d_i}(k-1)^{n-\ell-d_i}\ .\label{eq:d5}
\end{align}

Next, we evaluate $\frac{d}{dk} \pi_G^{(P_2)}(k)$ at $k=2$ and analyze the effect on the terms \eqref{eq:d1}-\eqref{eq:d5}. First, any term with a positive power of $k-2$ factor will vanish. Hence, $\eqref{eq:d3}=\eqref{eq:d5}=0$. Moreover, $\eqref{eq:d2}=0$ due to $\pi_{C_{\ell}}^{(P_2)}(2)=0$. It remains to examine $\eqref{eq:d1}$ and $\eqref{eq:d4}$. 

Let $\ell_2$ denote the number of cycle vertices with degree 2 (that is, cycle vertices whose only neighbors are on the cycle) and $n_1$ denote the number of vertices of degree 1, which are necessarily non-cycle vertices. When $k=2$, the only surviving terms in \eqref{eq:d1} are the $\ell_2$ summands with $d_i=2$. Similarly, the only surviving terms in \eqref{eq:d4} are the $n_1$ summands with $d_i=1$. Thus, 
$$
\frac{d}{dk} \pi_{G}^{(P_2)}(k)\Big|_{k=2} = \frac{1}{\ell} \frac{d}{dk} \pi_{C_{\ell}}^{(P_2)}(k) \Big|_{k=2} \cdot \ell_2 + \frac{1}{2} \pi_{C_{\ell}}(2)\cdot n_1 
$$
Using $\frac{d}{dk} \pi^{(P_2)}_{C_{\ell}}(k)\Big|_{k=2} = (-1)^{\ell} \cdot \ell$ and $\pi_{C_{\ell}}(2) = 1 +(-1)^{\ell}$, we reach our desired conclusion:
$$
\frac{d}{dk} \pi_{G}^{(P_2)}(2) = (-1)^{\ell} \cdot \ell_2 + \frac{1}{2}(1+(-1)^{\ell}) \cdot n_1 = 
\begin{cases} \ell_2+n_1 & \ell\text{ even} \\ -\ell_2 & \ell\text{ odd}\end{cases}\ .
$$

\section{Determining How Many Colors are 
 Needed For Hypercubes}\label{app:hypercubes}

Here we use Theorem~\ref{thm:hypercubes} to show that for the two graphs in Figure~\ref{fig:G7s} (reproduced here in Figure~\ref{fig:G7-graphs-only}), assessing the number of colors needed for hypercubes of dimension 0 through $n$ to appear in the respective coloring graphs cannot help distinguish the graphs.

\tikzset{
G7aALPH/.pic = {
\node (a) at (0,0) {};
\node (b) at (3,0) {};
\node (c) at (3,3) {};
\node (d) at (0,3) {};
\node (e) at (1.5,0.5) {};
\node (f) at (1.5,1.5) {};
\node (g) at (1.5,2.5) {};
\draw[style=thick] (a)--(b)--(c)--(d)--(a)--(e)--(f)--(g);
\draw[style=thick] (a)--(g)--(c);
\draw[style=thick] (b)--(g)--(d);
\draw[style=thick] (a)--(f);
\draw[radius=.35,fill=white](a) circle node{B};
\draw[radius=.35,fill=white](b) circle node{A};
\draw[radius=.35,fill=white](c) circle node{B};
\draw[radius=.35,fill=white](d) circle node{A};
\draw[radius=.35,fill=white](e) circle node{C};
\draw[radius=.35,fill=white](f) circle node{A};
\draw[radius=.35,fill=white](g) circle node{C};
}
}

\tikzset{
G7bALPH/.pic = {
\node (a) at (0,0) {};
\node (b) at (3,0) {};
\node (c) at (3,3) {};
\node (d) at (0,3) {};
\node (e) at (1.5,0.5) {};
\node (f) at (2.5,1.5) {};
\node (g) at (1.5,1.5) {};
\draw[style=thick] (a)--(b)--(c)--(d)--(a)--(e)--(b)--(f)--(c);
\draw[style=thick] (a)--(g)--(c);
\draw[style=thick] (b)--(g)--(d);
\draw[radius=.35,fill=white](a) circle node{C};
\draw[radius=.35,fill=white](b) circle node{B};
\draw[radius=.35,fill=white](c) circle node{C};
\draw[radius=.35,fill=white](d) circle node{B};
\draw[radius=.35,fill=white](e) circle node{A};
\draw[radius=.35,fill=white](f) circle node{A};
\draw[radius=.35,fill=white](g) circle node{A};
}
}

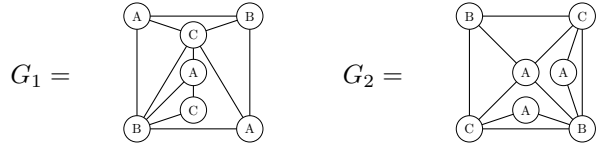
\begin{figure}[h]
\begin{tabular}{rlcrlc}
 \begin{tabular}{r}$G_1=$\end{tabular} & \begin{tabular}{l}\tikz \node [scale=0.5, inner sep=0] {
   \begin{tikzpicture}
     \pic at (0,0) {G7aALPH};
   \end{tikzpicture}
};\end{tabular} && 
    \begin{tabular}{r}$G_2=$\end{tabular} & \begin{tabular}{l}\tikz \node [scale=0.5, inner sep=0] {
   \begin{tikzpicture}
     \pic at (0,0) {G7bALPH};
   \end{tikzpicture}
};\end{tabular}
\end{tabular}
\caption{Two 7-vertex graphs each partitioned into three independent sets}\label{fig:G7-graphs-only}
\end{figure}

\begin{claim} $Q_0$ first appears in $k$-coloring graphs at $k=3$ of $G_1$ and $G_2$. 

More precisely, 
$\pi_{G_1}^{(Q_0)}(2)=\pi_{G_2}^{(Q_0)}(2)=0$, and $\pi_{G_1}^{(Q_0)}(3)>0, \pi_{G_2}^{(Q_0)}(3)>0$.
\end{claim}

Each graph has a triangle, so 2-colorings do not exist, yet the independent set partitions labeled in Figure~\ref{fig:G7-graphs-only} illustrate a 3-coloring of each graph, giving rise to $Q_0$ in the 3-coloring graphs.

\begin{claim} $Q_1, Q_2, Q_3$ first appear in $k$-coloring graphs of $G_1$ and $G_2$ at $k=4$. 

More precisely, $\pi_{G_1}^{(Q_1)}(3)=\pi_{G_2}^{(Q_1)}(3)=0$, and $\pi_{G_1}^{(Q_3)}(4)>0, \pi_{G_2}^{(Q_3)}(4)>0$.
\end{claim}

Next we note that the 3-coloring graphs of each graph are edgeless (i.e., they contain no copies of $Q_1$) because every vertex in each graph participates in a triangle, so any 3-coloring of the triangle leaves no available second color for any vertex. However, each graph has an independent set of size 3 (the A vertices) and the rest of the graph is 2-colorable; so $Q_3$ first appears in the 4-coloring graph for $G_1$ and $G_2$ by Corollary~\ref{cor:hypercube}. 

\begin{claim} $Q_4$ and $Q_5$ first appear in $k$-coloring graphs of $G_1$ and $G_2$ at $k=5$. 

More precisely, $\pi_{G_1}^{(Q_4)}(4)=\pi_{G_2}^{(Q_4)}(4)=0$, and $\pi_{G_1}^{(Q_5)}(5)>0, \pi_{G_2}^{(Q_5)}(5)>0$.
\end{claim}

Any $Q_4$ that arises in the 4-coloring graphs must correspond to 4 vertices changing color in $Q_4$, and no two of these vertices can be in the same triangle, because two vertices changing color in a triangle require 5 colors. Hence any color-changing vertex $v$ cannot have more than 3 neighbors, because this would leave at most $7-1-4=2$ additional candidate vertices to change color in $Q_4$. 
The subgraph of $G_1$ consisting of the five vertices of degree $\le 3$ contains no independent set of size 4, and $G_2$ only has three vertices of degree $\le 3$, so neither 4-coloring graph contains $Q_4$.

On the other hand, we show that both 5-coloring graphs contain $Q_5$. Using the independent partitions of $G$ into $A$, $B$, $C$, we let the vertices in $A$ take colors $1$ and $2$, vertices in $B$ take colors $3$ and $4$, and vertices in $C$ take color $5$. This shows $\pi^{(Q_5)}_{G_1}(5) > 0$ and $\pi^{(Q_5)}_{G_2}(5) > 0$. 

\begin{claim} $Q_6$ and $Q_7$ first appear in $k$-coloring graphs of $G_1$ and $G_2$ at $k=6$. 

More precisely, $\pi_{G_1}^{(Q_6)}(5)=\pi_{G_2}^{(Q_6)}(5)=0$, and $\pi_{G_1}^{(Q_7)}(6)>0, \pi_{G_2}^{(Q_7)}(6)>0$.
\end{claim}

In order for $Q_6$ to arise in a $5$-coloring graph, all but one of the vertices in the base graph must change color in $Q_6$. With only five colors, no choice of six vertices can include an entire triangle. Hence, the omitted vertex must participate in every triangle, and neither graph has such a vertex.

Both 6-coloring graphs contain $Q_7$ by assigning colors $1$ and $2$ to vertices in $A$, colors $3$ and $4$ to vertices in $B$, and colors $5$ and $6$ to vertices in $C$. 

\begin{claim} $Q_d$ never appears in any $k$-coloring graph of $G_1$ or $G_2$ for $d>7$.
\end{claim}

Indeed, $Q_d$ would require $d$ color-changing vertices which is not possible in either base graph for $d>7$.

As noted in Subsection~\ref{sec:hyper}, although both the chromatic pairs polynomial and this coarse invariant fail to distinguish the graphs, the counts of their 2-cubes (4-cycles) differ. Their distinct $C_4$-polynomials are given below for completeness: 

\begin{claim} 
    $G_1$ and $G_2$ have distinct $Q_2$-polynomials. Specifically, 
    \begin{align*}
    \pi_{G_1}^{(C_4)}(k) &= \frac{21}{4} k^9 - \frac{225}{2} k^8 + \frac{2103}{2} k^7 - \frac{22321}{4} k^6 + \frac{36637}{2} k^5 - \frac{151647}{4} k^4 + \frac{192091}{4} k^3 - \frac{67543}{2} k^2 + 9978 k \\
    \pi_{G_2}^{(C_4)}(k) &= \frac{21}{4} k^9 - \frac{225}{2} k^8 + \frac{2103}{2} k^7 - \frac{22323}{4} k^6 + \frac{36649}{2} k^5 - \frac{151757}{4} k^4 + \frac{192331}{4} k^3 - \frac{67667}{2} k^2 + 10002k
\end{align*}
\end{claim}

Noting that $Q_2=C_4$, these polynomials were constructed by applying Eq.~\eqref{eq:C4} to the two graphs. Interestingly, these polynomials agree on their first nonzero value of $k$: $\mathcal C_4(G_1)$ and $\mathcal C_4(G_2)$ each have 288 squares. They first differ at $k=5$: $\mathcal C_5(G_1)$ has 24540 squares and $\mathcal C_5(G_2)$ has 24360.

\end{document}